\documentclass[11pt]{amsart}

\usepackage{cite}
\usepackage{algorithmic}
\usepackage{graphicx}
\usepackage{textcomp}
 \usepackage{bbm}
\usepackage{hyperref}
\usepackage{cleveref}[2012/02/15]

\usepackage[margin=3cm]{geometry}

\newtheorem{lemma}{Lemma}
\newtheorem{theorem}[lemma]{Theorem}
\newtheorem{corollary}[lemma]{Corollary} 
\newtheorem{proposition}[lemma]{Proposition} 
\newtheorem{definition}[lemma]{Definition} 
\newtheorem{remark}[lemma]{Remark} 
\newtheorem{example}[lemma]{Example} 

\usepackage{tikz}
\usepackage{mathdots}
\usepackage{yhmath}
\usepackage{cancel}
\usepackage{color}
\usepackage{siunitx}
\usepackage{array}
\usepackage{multirow}
\usepackage{amssymb}
\usepackage{gensymb}
\usepackage{tabularx}
\usepackage{extarrows}
\usepackage{booktabs}
\usetikzlibrary{fadings}
\usetikzlibrary{patterns}
\usetikzlibrary{shadows.blur}
\usetikzlibrary{shapes}

\usepackage{helvet,bbm}

\newenvironment{Aenum}
{\begin{enumerate}

}
 {\end{enumerate}}

\renewcommand{\epsilon}{\varepsilon}

\newcommand{\ctr}[2]{ {\mathfrak u}_{#1#2}}

\newcommand{\lin}{L^{1}_{\mathrm {loc}}(\R^+;[0,+\infty))}
\newcommand{\R}{\mathbb{R}}
\newcommand{\N}{\mathbb{N}}
\newcommand{\abs}[1]{\left|#1\right|}
\newcommand{\norm}[1]{\left\lVert#1\right\rVert}
\newcommand{\base}[1]{\widehat{\mathrm e}_{#1}}

\crefformat{footnote}{#2\footnotemark[#1]#3}

\begin{document}

\title[Consensus in Multiagent Systems under communication failure]{Consensus of Multiagent Systems\\ under communication failure}
 \author[M.~Bentaibi]{Mohamed Bentaibi}\thanks{M.~Bentaibi is at d-fine s.r.l., via Giuseppe Mengoni 4, 20121 Milano, Italy (e-mail: mohamed.bentaibi@d-fine.com).}
 \author[L.~Caravenna]{Laura Caravenna}\thanks{L.~Caravenna is at the Department of Mathematics `Tullio Levi-Civita', Università degli Studi di Padova, Via Trieste 63, 35121 Padova, Italy (e-mail: laura.caravenna@unipd.it).}
  \author[J.-P.~A.~Gauthier]{Jean-Paul A. Gauthier}\thanks{J.-P.~A.~Gauthier is at the University of Toulon, France (e-mail: jean-paul.gauthier@univ-tln.fr)}
   \author[F.~Rossi]{Francesco Rossi}\thanks{F.~Rossi is at Università Iuav di Venezia, Dipartimento di Culture del Progetto
 Ca’ Tron, S. Croce 1957,
 30135 Venezia, Italy (e-mail: francesco.rossi@iuav.it).}
 \thanks{L.C and F.R. are members of the Gruppo Nazionale per l'Analisi Matematica, la Probabilit\`a e le loro Applicazioni (GNAMPA) of the Istituto Nazionale di Alta Matematica (INdAM).
 This research was supported by PRIN
 2020 ``Nonlinear evolution PDEs, fluid dynamics and transport equations: theoretical foundations and applications'' and PRIN PNRR P2022XJ9SX of the European Union – Next Generation EU}
 

\begin{abstract}
We consider multi-agent systems with cooperative interactions and study the convergence to consensus in the case of time-dependent connections, with possible communication failure.

 We prove a new condition ensuring consensus: we define a graph in which directed arrows correspond to connection functions that converge (in the weak sense) to some function with a positive integral on all intervals of the form $[t,+\infty)$. If the graph has a node reachable from all other indices, i.e.~``globally reachable'', then the system converges to consensus. We show that this requirement generalizes some known sufficient conditions for convergence, such as Moreau's or the Persistent Excitation one.
We also give a second new condition, transversal to the known ones: total connectedness of the undirected graph formed by the non-vanishing of limiting functions.

\noindent MS Codes: 68Q25, 68R10, 68U05


\end{abstract}

\maketitle



\tableofcontents

\section{Introduction}\label{S:intro}

The study of multi-agent interacting systems is crucial in control theory, both for intrinsic theoretical interests and for the numerous applications, see e.g.~\cite{motsch2011new,ayi2021mean,haskovec2013flocking,benedetto2022mean,shvydkoy2018topological,choicucker,moshtagh2007distributed,olfati2006flocking,ben2005opinion,bullo2009distributed,albi2020mathematical,jabin2014clustering}. One of the main issues is the problem of {\it consensus}, i.e.~of verifying or ensuring that all agents reach a common value, see e.g.~\cite{bonnet2021consensus,bonnet2022consensus,motsch2014heterophilious,choi2021consensus,olfati2007consensus,nedich2015convergence,valentini2017achieving,jadbabaie2003coordination,ren2008distributed,anderson2016convergence,cicolani2024first,continelli2023convergence,moreau2004stability}. This is the problem that we address in this article.

One of the open problems for multi-agent systems is to understand their behaviour under communication failure. It has been studied in many contributions, see e.g.~\cite{tang2020bearing,manfredi2016criterion,bonnet2021consensus, cicolani2024first,continelli2023convergence}. Among them, an interesting line of contributions focuses on sufficient conditions that ensure consensus. A typical example is the condition introduced by Moreau in~\cite{moreau2004stability}, which is a generalization of the so-called {\it persistent excitation}, see e.g.~\cite{ABR-PE,bonnet2024exponential,ChSi2010,ChSi2014,tang2020bearing,chowdhury2018estimation,shi2013role}: if connections between agents are activated for a sufficient amount of time and on a network with a suitable structure, then consensus occurs. We discuss it in detail in~\S~\ref{s-PE}. Another very relevant condition, introduced by Hendrickx and Tsitsiklis, is called {\it the cut-balance} assumption, see~\cite{martin2013continuous,hendrickx2012convergence}. We will discuss it in detail in~\S~\ref{s-cutbalance}. The main result of our article is to provide two new conditions ensuring convergence of multi-agent systems. We {{}prove} that such conditions generalize the Moreau condition, and {{}show through examples} that our analysis and results are transversal to the cut-balance assumption: {{}there are situations where our conditions ensure consensus convergence while the cut-balance assumption does not hold, but also opposite cases where our result cannot be applied while the cut-balance assumption works}.\\

More in detail, we consider the system, for $j=1,\dots,N$,
	\begin{align}\label{E:basicsystem}
	&\dot x_{j}=\sum_{k=1}^{N} \ctr{j}{k}(t)\left(x_{k}-x_{j}\right)\,,
	&&\text{where } \ctr{j}{k}{{}(t)}\geq0{}\text{ for a.e.~} t>0.
	\end{align}
 It is a linear system of $N$ agents in $\R^d$, indexed by $j$, that interact with a cooperative rule. The influence of agent $k$ on agent $j$ is given by the function $\ctr{j}{k}:[0,+\infty)\to \R$ that we assume to be integrable on compact intervals. We highlight that interactions are time-dependent functions that do not depend on the state. By the cooperative rule, see~\cite{smith}, we mean that all components of the Jacobian $\partial_{k}\dot x_j$ are nonnegative for $k\neq j$, thus $\ctr{j}{k}\geq0$ in case of~\eqref{E:basicsystem}.
 
 In this model, when $0\leq \ctr{j}{k}(t)\leq 1$, the idea is that the full connection is given by $\ctr{j}{k}=1$, while lower values model communication failure. For full connection, it is easy to prove that, for any initial configuration of $x_j$, the system converges to {\it consensus}: there exists a common value $x^*$ such that $\lim_{t\to+\infty}x_j(t)=x^*$ for all $j$. The main question of this article is the following:\\

\noindent {\it {\bf Question:} Which ``minimal''  {{}properties} on the $\ctr{j}{k}$ guarantee that the system converges to {\it consensus} for any initial condition?}\\

\emph{This question can be seen as a request of minimal level of service to ensure consensus}. It has been extensively studied in the community. The contributions that are closer to our approach are the following:
 \begin{itemize}
 \item {\bf Moreau condition:} In~\cite{moreau2004stability}, Moreau introduces a condition for linear systems ensuring convergence, based on defining a graph: for some fixed $\mu$, an arrow from agent $j$ to agent $k$ is built if the connection function satisfies $$\int_t^{t+T} \ctr{j}{k}(s)\,ds\geq \mu>0$$ for all $t\geq 0$ and some $T>0$. If $\ctr{j}{k}$ are bounded and the resulting graph has a node that can be reached from all other nodes, i.e.~``globally reachable'', then the system exponentially converges to consensus. Associated estimations of the rate of convergence can be found in~\cite{chowdhury2018estimation}. In~\cite{chowdhury2015consensus}, the case of second-order systems is tackled. More restrictive conditions, known as Persistent Excitation or Integral Scrambling Coefficients, are also introduced and discussed in~\cite{bonnet2024exponential, ABR-PE,ChSi2010,ChSi2014}.
 \item {\bf Cut-balance:} In~\cite{hendrickx2012convergence}, the cut-balance condition assumes that $ \int_0^T\ctr{k}{j}(t)<+\infty$ for all $T>0$ and that there exist a constant $K>0$ such that for all subsets of agents $S \subset \{1,\ldots,N\}$ and for all $t>0$ it holds 
 \begin{align*}
 \sum_{j \in S, k \notin S} \ctr{j}{k}(t) \leq K \sum_{j \in S, k \notin S} \ctr{k}{j}(t).
 \end{align*} In~\cite{shi2013role}, a generalization, known as the arc-balance condition, is introduced. In~\cite{martin2015continuous}, the result is extended to allow for non-instantaneous reciprocity. This is one of the best available results in the literature, to our knowlarrow: we compare it to our contributions in~\S~\ref{s-cutbalance}. We also recall that in~\cite{martin2013continuous,martin2015continuous} the Persistent Excitation condition and the cut-balance condition are combined.
\end{itemize} 
 
Our main theorems provide two conditions that are new with respect to the ones described above, and have weaker hypotheses with respect to many of them. Moreover, we will show that these requirements are somehow sharp, in the sense that outside the hypotheses of the theorems it is easy to find examples for which consensus is not achieved. To describe our result, we first need the following easy definition.

\begin{definition}[Globally reachable node]
A node $\ell^*$ of a graph $G$ is ``globally reachable'' if for all nodes $i$, there exists a path of arrows $i\to j_1\to \ldots \to \ell^*$.
\end{definition}
This concept was already stated in~\cite{moreau2004stability} as a key property of graphs ensuring consensus, and it ensures that the directed graph contains a directed spanning tree.

We now define the topology for the connection functions, that we explain in~\S~\ref{Ss:topology}.

\begin{definition} \label{D:conver}
{{}Let $f_{n},f:[0,+\infty)\to[0,+\infty)$ be} Lebesgue integrable in compact intervals, for $n\in\N$. We say that $f_{n}\overset{\ast}{\rightharpoonup} f$ if
\[
\lim_{n\to\infty}\int_{a}^{b}f_{n}=\int_{a}^{b}f
\qquad\text{for all bounded intervals $[a,b]\subset [0,+\infty)$.}
\]
\end{definition}

\begin{remark}\label{rem-weakstar} In the most common case, with bounded connection functions, the topology above reduces to the weak$^{*}$-topology of $L^{\infty}$ as the dual of $L^{1}$, see Lemma~\ref{L:topology}.
\end{remark}

Our first main result for the article is the following.
\begin{theorem}\label{T:asymmetric}
Let ${}\ctr{j}{k},\ctr{j}{k}^{*}:[0,+\infty)\to[0,+\infty)$ be Lebesgue integrable in compact intervals for $j,k=1,\dots,N$. {{}Let $t_n\to+\infty$ be a sequence} such that, for each $j,k=1,\dots,N$, the function $f_{n}(t):= \ctr{j}{k}(t_{n}+t)$ converges as in Definition~\ref{D:conver} to the limit function $\ctr{j}{k}^{*}$. Define the directed graph {{}$G=G(\{t_n\},\{\ctr{j}{k}^{}\})=G(\{\ctr{j}{k}^{*}\})$ where}:
\begin{itemize}
\item nodes are {{}identified with} $\{1,\dots,N\}$;
\item we draw an arrow from node $j$ to node $k$ if the following holds:
\begin{equation}\label{E:connectionAsymmetric}
\int_{t}^{+\infty}\ctr{j}{k}^{*}>0
\qquad
\forall t>0.
\end{equation}

\end{itemize}
Assume that the directed graph $G={}G(\{t_n\},\{\ctr{j}{k}^{}\})$ has a globally reachable node. Then, for all initial configurations, the solutions of~\eqref{E:basicsystem} converge to consensus.
\end{theorem}



We discuss and prove this first result in~\S~\ref{s-main} {and {}\S~\ref{S:ex} contains many examples.}
Via the following, simpler but more restrictive, corollary, we already show that the condition in Theorem~\ref{T:asymmetric} is much weaker than the Moreau condition~\cite{moreau2004stability}. See a more detailed comparison in~\S~\ref{s-PE}.

\begin{corollary}\label{c-1} 
{{}Let $\ctr{j}{k}:[0,+\infty)\to[0,+\infty)$ be} Lebesgue measurable and boun\-ded, for $j,k=1,\dots,N$. Define the directed graph {{}$G=G(\{\ctr{j}{k}\})$ where}:
\begin{itemize}
\item nodes are {{}identified with} $\{1,\dots,N\}$;
\item we draw an arrow from node $j$ to node $k$ if one of the following (equivalent) {{}properties} hold:
\begin{Aenum}
\item \label{item:1} 
$\displaystyle{\limsup_{T\to+\infty}\,\liminf_{t\to+\infty}\int_{t}^{t+T}\ctr{j}{k}>0}$.
\item \label{item:2} There exist $T,\mu>0$ such that for all $t\geq 0$ it holds
$\displaystyle{\int_{t}^{t+T} \ctr{j}{k}\geq \mu
}$.
\item \label{item:3} There exist $T,\mu>0$ and a sequence $t_n\to+\infty$ with $\{ {t_{n+1}}-t_{n}\}_{n\in\N}$ bounded such that
$\displaystyle{
\int_{t_n}^{t_n+T} \ctr{j}{k}\geq\mu}$ for all $ n\in\N$ {{} and all $j,k=1,\dots,N$}.
\end{Aenum}
\end{itemize}
Assume that the directed graph $G={}G(\{\ctr{j}{k}\})$ has a globally reachable node. Then, for all initial configurations, solutions of~\eqref{E:basicsystem} converge to consensus.
\end{corollary}
{{}
The equivalence of properties \ref{item:1}-\ref{item:2}-\ref{item:3} is proved in Lemma~\ref{L:equivalentConditions} below.
We observe the following interesting phenomenon, which is one of the key sharpness results of our article: Example~\ref{ex-basic2} below shows a case in which {{}property (C) fails just because} $t_{n+1}-t_n$ slowly grows like $\log(n)$, thus is not bounded, and consensus is not achieved.}

\begin{remark}[Sufficient number of connections]\label{rem-BA}
Suppose that the connection functions $ \ctr{j}{k}$ are all bounded.
Suppose, for a suitable sequence $t_n$, one draws enough arrow with {{}property}~\eqref{E:connectionAsymmetric} \emph{only} to establish that a node in the directed graph $G$ is globally reachable.
Then, nothing more has to be done to apply Theorem~\ref{T:asymmetric}: for a suitable subsequence $t_{n_i}$, due to Remark~\ref{rem-weakstar} and by the Banach-Alaoglu theorem, also the remaining coefficients $ \ctr{j}{k}$  automatically converge to some limit functions $ \ctr{j}{k}^*$ (due to boundedness). Whether these remaining limit functions $ \ctr{j}{k}^*$ satisfy~\eqref{E:connectionAsymmetric} or not will play no role, since the existence of a globally reachable node is already established, see Remark~\ref{R:eccoloqui} below.
\end{remark} 

The second main result of this article is stated similarly to Theorem~\ref{T:asymmetric}, but the request on the graph $G$ is different. It is as follows:
\begin{theorem}\label{T:asymmetricBis}
Let ${}\ctr{j}{k},\ctr{j}{k}^{*}:[0,+\infty)\to[0,+\infty)$ be Lebesgue integrable in compact intervals, for $j,k=1,\dots,N$. Let $t_n\to+\infty$ be such that, for $j,k=1,\dots,N$, the sequence of functions $f_{n}(t):=  \ctr{j}{k}(t_{n}+t)$ converges as in Definition~\ref{D:conver} to the limit function $\ctr{j}{k}^{*}$. {{}Construct the directed graph $G={}G(\{t_n\},\{\ctr{j}{k}^{}\})=G(\{\ctr{j}{k}^{*}\})$ where}:
\begin{itemize}
\item nodes are {{}identified with} $\{1,\dots,N\}$;
\item we draw an arrow from node $j$ to node $k$ if the following holds: 
\begin{equation}\label{E:connectionAsymmetric2}
\int_{0}^{+\infty}\ctr{j}{k}^{*}>0
\,.
\end{equation}
\end{itemize}
Assume that for each pair $j,k$ there exists at least one arrow from node $j$ to node $k$ or from $k$ to $j$. Then, for all initial configurations, solutions of~\eqref{E:basicsystem} converge to consensus.
\end{theorem}

Observe that, in this case, the direction of arrows plays no role. On the opposite, a very large number of connections is required; nevertheless, connections are easier to establish, since we just require that the limiting function is non-vanishing.

\begin{remark}\label{R:general}
 Even though the dynamics in~\eqref{E:basicsystem} is chosen to be linear in the state variables, all our results can be restated for nonlinear systems of the form
  \begin{equation}\label{E:non-lin}
	\dot x_{j}=\sum_{k=1}^{N} \ctr{j}{k}(t,x)(x_{k}-x_{j})\qquad\qquad 
	j=1,\dots,N\,,
	\end{equation}
 where $\ctr{j}{k}$ are bounded and $\ctr{j}{k}(t,x)\geq \ctr{j}{k}^-(t)$, for functions ${\ctr{j}{k}^-}$ that satisfy the hypotheses of our theorems.
 We provide details in Propositions~\ref{p:non-lin}-~\ref{p:non-lin2} below.
\end{remark}

%

The structure of the article is as follows:
\begin{itemize}
\item[\S~\ref{S:cooperativeMAS}:] We state general results about systems of the form~\eqref{E:basicsystem}.
\item[\S~\ref{s-main}:] We prove Theorem~\ref{T:asymmetric}, Corollary~\ref{c-1} and Theorem~\ref{T:asymmetricBis}.
\item[\S~\ref{S:exComp}:] We compare our results with the literature. Several examples show that our conditions are new and either more general or transversal to the known ones.
\end{itemize} 

\section{Cooperative multi-agent systems}\label{S:cooperativeMAS}

In this section, we describe some general properties of cooperative multi-agent systems. In this article, we only deal with one-to-one interactions, but we consider possible communication failure {{}in the following sense: we provide conditions ensuring convergence even when many agents can stop communicating for large intervals of time}. In wide generality, we study systems of the following form: 
\begin{equation}
\label{E:nonlin}
	\dot x_{j}=\sum_{k=1}^{N} \ctr{j}{k}(t)\phi\left(x_{k}-x_{j}\right) (x_{k}-x_{j}),\qquad
	 \ j=1,\dots,N\,,
 \end{equation}
 where $\ctr{j}{k}\in \lin$ and $\phi$ is nonnegative, bounded and Lipschitz continuous. We denoted by $\lin$ the functional space
 \begin{equation}
 \label{linloc}\left\{f:\R^+\to[0,+\infty)\text{ Lebesgue measurable with $\int_0^T f<+\infty$ when $T>0$}\right\}.
 \end{equation}
 This ensures existence, globally in time, and uniqueness for the solution to the associated Cauchy problem, i.e.~when an initial condition $(x_1(0),\ldots, x_N(0))$ is fixed, see e.g.~\cite{filippov}. Solutions are considered in the Carathéodory sense for the rest of the article: trajectories are absolutely continuous functions and~\eqref{E:nonlin} holds at almost every time. 
 General results on cooperative systems can also be found in~\cite{smith}.
Now: 
\begin{itemize}
\item[\S~\ref{S:reduction}:] We reduce to the case of $1$-dimensional, linear systems.
\item[\S~\ref{S:general}:] We remind that the convex hull of positions is weakly contractive in time, and we {{}discuss monotonicity of the set of agents attaining extremal values}.
\item[\S~\ref{Ss:topology}:]We better explain the topology involved in our sufficient conditions.
\end{itemize}

\subsection{Reduction to 1-dimensional linear systems}\label{S:reduction}
 In our article, we study convergence to consensus for~\eqref{E:nonlin} by considering all possible connection functions $\ctr{j}{k}(t)$, under the assumption that they are integrable on compact intervals and nonnegative. As a consequence, it is not restrictive to assume that the dynamics is linear, as we stated in~\eqref{E:basicsystem} in the introduction. In fact, we have the following simple results.

 \begin{proposition}\label{p:non-lin} 
 Consider a function $\phi$, bounded on compact intervals, and connections $\ctr{j}{k}\in\lin$ as in~\eqref{linloc}, for $j,k=1,\dots,N$.
 Consider any given solution $x(t)$ to~\eqref{E:nonlin} starting from a fixed initial condition $(x_1(0),\ldots, x_N(0))$. Then there exist functions $ \widetilde{\ctr{j}{k}}\in\lin$ such that $x(t)$ solves the linear system
 \begin{equation}
\label{E:lin-fix}
	\dot x_{j}=\sum_{k=1}^{N} \widetilde{\ctr{j}{k}}(t)\left(x_{k}-x_{j}\right),\qquad\qquad \ j=1,\dots,N\,.
 \end{equation}
 If $\ctr{j}{k}$ are bounded on compact intervals and $M:=\max_{[0,T]}\norm{x(t)}$, it holds \[\norm{ \widetilde{\ctr{j}{k}}}_{L^{\infty}[0,T]}\leq \norm{\ctr{j}{k}}_{L^{\infty}[0,T]}\cdot \norm{\phi}_{L^{\infty}[0,M]}\,.\]
\end{proposition}
\begin{proof}
 Consider any given trajectory $x(t)$ of~\eqref{E:nonlin} and assume that $t$ is a time for which $x$ is differentiable. Then it clearly holds $$\dot x_j=\sum_{k=1}^{N} \ctr{j}{k}(t)\phi\left(x_{k}-x_{j}\right) (x_{k}-x_{j})=
 \sum_{k=1}^{N} \widetilde{\ctr{j}{k}}(t)\left(x_{k}-x_{j}\right),$$
 by choosing $$\widetilde{\ctr{j}{k}}(t):=\ctr{j}{k}(t)\phi\left(x_{k}(t)-x_{j}(t)\right)\,.
 $$
 Such coefficients $ \widetilde{\ctr{j}{k}}$ are integrable on compact intervals: in any interval $[0,T]$ indeed\[
 \widetilde{\ctr{j}{k}}(t)\leq C {\ctr{j}{k}}(t)\,,\quad C:=\norm{\phi}_{L^{\infty}[0,M]}\,,\quad
 M:=\norm{x}_{L^{\infty}[0,T]}.\]
 \end{proof}

 \begin{proposition}\label{p:non-lin2} 
 Let $M>0$.
 Consider functions $\ctr{j}{k}:\R^+\times\R^N\to[0,M]$, for $j,k=1,\dots,N$, that are measurable for all continuous Borel probability measures, i.e.~``universally measurable''.\footnote{{}Since universally measurable functions are closed under composition \cite[Proposition 7.44]{meas}, the measurability of $(t,x)\mapsto \ctr{j}{k}(t,x)$ is a standard condition to ensure that $t\mapsto \ctr{j}{k}(t,x(t))$ is Lebesgue measurable.
 We recall that the $\sigma$-algebra $\mathcal U$ of universally measurable sets is defined as the intersection, over all Borel probability measure $p$ on $\R^{n}$, of the $\sigma$-agebra of $p$-measurable sets and we recall that a function $f:\R^{n}\to\R$ is universally measures if $f^{-1}(I)\in\mathcal U$ for all intervals $I\subset \R$.}
 Consider any given solution $\overline x:\R^+\to\R^{N}$ to~\eqref{E:non-lin} starting from a fixed initial condition $\overline x_0=(\overline x_1(0),\ldots, \overline x_N(0))$.
 Suppose connections
 \[
 \ctr{j}{k}^-(t):=\inf\left\{\ctr{j}{k}(t,x) \ : \ \norm{x-\overline x_0}\leq M\sqrt N \cdot t\right\}
 \]
satisfy assumptions of Theorems~\ref{T:asymmetric}, or Corollary~\ref{c-1}, or Theorem~\ref{T:asymmetricBis}.
 Then the trajectory $\overline x(t)$ reaches consensus: $ \overline x(t)\to(\overline x^*,\dots,\overline x^*)$ as $t\to+\infty$, for some $\overline x^*\in\R$.
\end{proposition}

\begin{proof}
Define 
$ \widetilde{\ctr{j}{k}}(t):={\ctr{j}{k}}(t,\overline x(t))$.
At any time $t$ when $\overline x$ is differentiable, then $\dot {\overline x}_j=\sum_{k=1}^N\widetilde{\ctr{j}{k}}(t)(\overline x_k-\overline x_j)$.
Notice that $\abs{\widetilde{\ctr{j}{k}}}\leq M$ and $\widetilde{\ctr{j}{k}}\geq \ctr{j}{k}^-$.
If ${\ctr{j}{k}^-}$ satisfies the hypothesis of Corollary~\ref{c-1}, then trivially the same holds for $\widetilde{\ctr{j}{k}}$ and we get the thesis.
Let now $t_k\to+\infty$ be a sequence of times when the connections $ t\mapsto\ctr{j}{k}^-(t_k+t)$ converge weakly$^*$ to limit functions $\ctr{j}{k}^{-*}$: consider the graph $G^-$ defined by condition~\eqref{E:connectionAsymmetric} relative to ${\ctr{j}{k}}^{-*}$.
By Banach-Alaoglu theorem, up to extracting a subsequence, $t\mapsto\widetilde{\ctr{j}{k}}(t_k+t)$ converge weakly$^*$ to limit functions $\widetilde{\ctr{j}{k}}^*$; in particular, since necessarily $\widetilde{\ctr{j}{k}}^*\geq \ctr{j}{k}^{-*}$ by properties of weak convergence, the graph $\widetilde G$ defined by condition~\eqref{E:connectionAsymmetric} relative to $\widetilde{\ctr{j}{k}}^*$ has all the arrows present in $G^-$.
By Lemma~\ref{L:topology}, we conclude that, if the coefficients ${\ctr{j}{k}}^{-}$ satisfy the assumptions of Theorem~\ref{T:asymmetric}, then also the $\widetilde{\ctr{j}{k}}$ do, reaching the thesis.
With Theorem~\ref{T:asymmetricBis} the argument is similar.
\end{proof}

Thanks to these simple results, from now on we will only consider the linear dynamics given in~\eqref{E:basicsystem}. 
We also aim to restrict ourselves to study 1-dimensional systems. This is the meaning of the following result.

\begin{proposition} \label{p-multid}
 Let $d\in\mathbb{N}$ and $v\in \R^d$. Consider a trajectory $$x(t)=(x_1(t),\ldots,x_N(t))$$ to~\eqref{E:basicsystem} with $x_j(t)\in\R^d$ starting from a fixed initial condition $(x_1(0),\ldots, x_N(0))$ and with connection functions $\ctr{j}{k}(t)$. Then, the projected trajectory $$y(t;v)=(y_1(t),\ldots,y_N(t))$$ with $y_j(t)\in \R$ defined by $y_j(t):=x_j(t)\cdot v$ is the unique solution to~\eqref{E:basicsystem} defined in $\R$ with projected initial data $y_j(0):=x_j(0)\cdot v$ and the same connection functions $\ctr{j}{k}$.

 In particular, the trajectory $x(t)$ converges to consensus if and only if, for all vectors $v\in \R^d$, the projected trajectory $y(t;v)$ converges to consensus.
\end{proposition}
 
 \begin{remark}{}
 One can recover $x$ from $d$ projections, writing
 $x(t)=\sum_{j=1}^{d}y(t,\base{j})\base{j}$, provided that $\base{j}\cdot\base{k}=\delta_{jk}$, $j,k=1,\dots,d$: when $\{\base{1},\dots,\base{d}\}$ is an orthonormal basis.
 \end{remark}

\begin{proof} We prove the first statement. Let $x(t)$ be a trajectory. At times $t$ for which $x$ is differentiable, by differentiating the identity $y_j(t)=x_j(t)\cdot v$, we have
\begin{equation}\label{e:connectionproj}
 \dot y_j=\sum_{k=1}^N\ctr{j}{k}(x_k(t)\cdot v-x_j(t)\cdot v)=\sum_{k=1}^N\ctr{j}{k}(y_k(t)-y_j(t)).
\end{equation}

We now prove the second statement. We first prove the first implication. Let $x(t)$ converge to a consensus, i.e.~$\lim_{t\to+\infty}x_j(t)=x^*$ for all $j=\{1,\ldots,N\}$. Let $v\in\R^d$. By continuity of the scalar product, it holds $\lim_{t\to+\infty}y_j(t)=\lim_{t\to+\infty}x_j(t)\cdot v=x^*\cdot v$ for all $j=\{1,\ldots,N\}$, thus $y(t;v)$ converges to consensus.

We now prove the reverse implication. Choose the standard basis $\base{1},\ldots,\base{d}$ of unitary vectors of $\R^d$, i.e.~$\base{\ell}=(0,\ldots,0,1,0,\ldots)$ with 1 in position $\ell$. For each $\ell=1,\ldots,d$ the variables $y_j(t)=x_j(t)\cdot \base{\ell}$ converge to consensus, i.e.~the $\ell$-th component of $x_j(t)$ converges to some $(x^\ell)^*$. Since this holds for all components, all $x_j(t)$ converge to the common vector $((x^1)^*,\ldots,(x^d)^*)$, i.e.~to consensus. 
\end{proof}

\subsection{General properties of cooperative systems}\label{S:general}
We now {{}collect} general properties of~\eqref{E:basicsystem}. Being cooperative, {{}in the time-independent case it is well known that} its support is (weakly) contractive, {{}see e.g.~\cite{contr}. We repeat the proof here for completeness in the time-dependent case, that is very similar}:

\begin{proposition} \label{p:contractive} Let $x(t)$ be a solution of~\eqref{E:basicsystem}. Define the support of the solution at time $t$ as the (closed) convex hull of the set of $x_i$ at time $t$: precisely
\begin{equation}\label{e-support}
\mathrm{supp}(x(t)):=\mathrm{conv}(\{x_i(t)\}),
\end{equation}  Then, for $0\leq t\leq s$ it holds $\mathrm{supp}(x(t))\supseteq \mathrm{supp}(x(s))$. 

In dimension $d=1$, this implies that the maximum function $x_+(t):=\max_j\{x_j(t)\}$ is non-increasing and the minimum function $x_-(t):=\min_j\{x_j(t)\}$ is non-decreasing.
\end{proposition}
\begin{proof} First observe that $\mathrm{supp}(x(t))$ is the convex hull of a finite number of points, hence it is a closed polygon.

Let $t$ be a time in which $x(t)$ is differentiable. If $x_j(t)$ belongs to the interior of $\mathrm{supp}(x(t))$, by continuity it belongs to the interior of $\mathrm{supp}(x(t+h))$ for $h>0$ sufficiently small. Assume then that $x_j(t)$ belongs to the boundary of $\mathrm{supp}(x(t))$: each term $\ctr{j}{k}(t)\left(x_{k}-x_{j}\right)$ points inwards in the polygon, due to the fact that $x_k$ belongs to the polygon and $\ctr{j}{k}(t)$ is positive. Then, the sum of all terms, that is $\dot x_j$, points inwards. Thus, one has $x_j(t+h)\in \mathrm{supp}(x(t))$ for $h>0$ sufficiently small. By merging the two cases, one has $x_j(t+h)\in \mathrm{supp}(x(t))$ for all $j=1,\ldots,N$, hence by convexity $\mathrm{supp}(x(t+h))\subseteq \mathrm{supp}(x(t))$. This proves the first result.

The results in dimension $d=1$ directly follow, since $\mathrm{supp}(x(t))$ is an interval.
\end{proof}

The last statement in dimension $d=1$ is very strong.
We even strengthen it, as follows, when extremal values are constant.

\begin{lemma}\label{L:belowmaximumpreserved}
Consider a trajectory $x(t)$ of~\eqref{E:basicsystem} in $\R$ such that $x_{+}^{*}=\max\{x_{i}(t)\ :\ i=1,\dots,N\}$ is constant.
Then the set $I^+(t)$ of indices $i$ that realize this maximum is non-increasing in time: if $i\notin I^{+}(t)$ then $i\notin I^{+}(t+h)$ for all $h>0$.

Similarly, assume that $x_{-}^{*}=\min\{x_{i}(t)\ :\ i=1,\dots,N\}$ is constant.
Then the set $I^-(t)$ of indices $i$ that realize this minimum is non-increasing in time.
\end{lemma}

\begin{proof} 
Consider an index $j\notin I^+(T)$ for some $T\geq 0$, which means $x_j(T)<x^*_+$. Define $f(t):=x^*_+-x_j(t)$, that satisfies $f(T)>0$. Let $t$ be a point in which $x(t)$ is differentiable. By the dynamic~\eqref{E:basicsystem} it holds
\begin{align*}
\dot f(t)&=0-\sum_{k=1}^{N} \ctr{j}{k}(t)\left(x_{k}-x_{j}(t)\right)
\geq -\sum_{k=1}^{N} \ctr{j}{k}(t)\left(x^*_+-x_{j}(t)\right)=-\sum_{k=1}^{N} \ctr{j}{k}(t) f(t).
\end{align*}
In the first inequality we used that $x_{k}\leq x^{*}_{+}$, for all $k=1\,,\dots\,,N$, by the definition of $x^{*}_{+}$ as a maximum.
Gronwall lemma now ensures $$f(t)\geq f(T)\cdot\exp\left(-\int_T^t \sum_{k=1}^{N} \ctr{j}{k}(s)\,ds\right)>0.$$
By continuity, the estimate holds for all $t\geq T$, ensuring that $j\not \in I^+(t)$ for all $t\geq T$.

The statement on the minimum can be proved analogously.
\end{proof}

We will use this simple result in Lemma~\ref{L:constantevolution} below, as well as in the proofs of Theorems~\ref{T:asymmetric} and~\ref{T:asymmetricBis}. We will indeed prove all the main statements in dimension 1, then by Proposition~\ref{p-multid} they hold in any dimension.

\subsection{The weak* topologies}
\label{Ss:topology}


In this section we prove a technical lemma to better understand the topology in Definition~\ref{D:conver}. We embed nonnegative functions, integrable on compact intervals, into the space of Radon measures, with the inherited weak$^{*}$-topology. When further restricting to nonnegative bounded functions, we get the weak$^{*}$-topology of $L^{\infty}$ as the dual of $L^{1}$.

\begin{lemma}{}\label{L:topology}
For $n\in\N$, let $f_{n}, f\in\lin$, defined in~\eqref{linloc}.
Then the convergence $f_{n}\overset{\ast}{\rightharpoonup} f$ specified in Definition~\ref{D:conver} is equivalent to
\begin{itemize}
\item $f_{n}$ converges to $ f$ if for all $\varphi:\R^+\to\R$ continuous with compact support
\begin{equation}\label{E:testing}
\lim_{n\to+\infty}\int_{0}^{+\infty} \varphi f_{n}=\int_{0}^{+\infty} \varphi f\,.
\end{equation}
\end{itemize}

If $f_{n}, f$ are nonnegative and bounded on compact intervals, it is equivalent to require~\eqref{E:testing} for all $\varphi:\R^+\to\R$ having compact support and with $\int_0^{+\infty}\abs{\varphi}$ finite.

If, moreover, $f_{n}, f:\R^+\to[0,M] $ for some $M>0$, for all $n\in\N$, {{}requiring~\eqref{E:testing} for all $\varphi:\R^+\to\R$ continuous with compact support is equivalent to requiring~\eqref{E:testing} for all $\varphi:\R^+\to\R$ with $\int_0^{+\infty}\abs{\varphi}$ finite: on $L^{\infty}$, the convergence $\overset{\ast}{\rightharpoonup} $ is the weak$^{*}$-topology.}
\end{lemma}

\begin{proof}%
The equivalence among Definition~\ref{D:conver} and the one in~\eqref{E:testing} follows from~\cite[Theorem~1.40]{evans-gariepy}, by regularity of Radon measures.

Suppose now additionally that $ f_{n} \leq M(C)$ and $ f \leq M(C)$ in $[0,C]$, for all $n\in\N$. 

If~\eqref{E:testing} holds for all $\varphi:\R^+\to\R$ having compact support and with $\int_0^{+\infty}\abs{\varphi}$ finite, then it trivially holds also for any $\varphi:\R^+\to\R$ continuous with compact support.

If~\eqref{E:testing} holds for any $\varphi:\R^+\to\R$ continuous with compact support, consider any $\psi:\R^+\to\R$ having compact support and with $\int_0^{+\infty}\abs{\psi}$ finite, and extend it to be $0$ on $\R^{-}$.
Let $\psi_{\varepsilon}$ be a smooth approximation in $L^{1}(\R)$ of $\psi$, having compact support in some $[0,C]$, for example by convolution, see~\cite[\S~4.2.1]{evans-gariepy}.
Then take the limsup first as $n\to+\infty$ then, as $\varepsilon\to0$, in the triangular inequality
\begin{align*}
\abs{\int_{0}^{+\infty}(f_{n}- f_{})\psi
}
&\leq
\abs{\int_{0}^{+\infty}(f_{n}- f_{})\psi_{\varepsilon }}
+
M(C)\norm{\psi -\psi_{\varepsilon}}_{L^{1}(\R)}
\end{align*}
to conclude that~\eqref{E:testing} holds also for $\psi$.

Suppose now additionally that $f_{n}$ and $f$ are uniformly bounded by $M$. If $C>0$ and $\varphi:\R^+\to\R$ has $\int_0^{+\infty}\abs{\varphi}$ finite, then $\psi_{C}=\varphi\mathbbm{1}_{[0,C]}$ has compact support: thus, by the previous step,~\eqref{E:testing} holds for $\psi_{C}$.
To conclude that~\eqref{E:testing} holds also for $\varphi$, take the limsup, first as $n\to+\infty$, then as $C\to+\infty$, in the inequality
\begin{align*}
\abs{\int_{0}^{+\infty}(f_{n}- f_{})\varphi
}
&\leq
\abs{\int_{0}^{+\infty}(f_{n}- f_{})\psi_{C }}
+
M\norm{\varphi}_{L^{1}((C,+\infty))}\,.
\end{align*}
\end{proof}

\section{Proof of main results} \label{s-main}

In this section, we focus on establishing the new sufficient conditions for consensus in~\eqref{E:basicsystem}, which constitute the main results of this paper: we prove Theorem~\ref{T:asymmetric} in \S~\ref{sec:proof1}, Corollary~\ref{c-1} in \S~\ref{s-equivalence}, Theorem~\ref{T:asymmetricBis} in \S~\ref{s-asym2}.

	\subsection{Proof of Theorem~\ref{T:asymmetric}} \label{sec:proof1} In this section, we prove Theorem~\ref{T:asymmetric}. We first prove an auxiliary lemma for the dynamics on the real line, extending Lemma~\ref{L:belowmaximumpreserved}.
 
\begin{lemma}\label{L:constantevolution}
Let $x(t)$ be a trajectory of~\eqref{E:basicsystem} in $\R$ with given connection functions $\ctr{j}{k}\in\lin$ as in~\eqref{linloc}, $j,k=1,\dots,N$. Assume that both \begin{align*}
 &
x_{+}^{*}=\max\{x_{i}(t)\ :\ i=1,\dots,N\} &&\text{and}
&&x_{-}^{*}=\min\{x_{i}(t)\ :\ i=1,\dots,N\}
\end{align*}
are constant.
Consider the graph  {{}$G=G(\{\ctr{j}{k}\})$ constructed} as follows:
\begin{itemize}
\item nodes are {{}identified with} $\{1,\dots,N\}$ and
\item we draw an arrow from node $j$ to node $k$ when
\begin{equation}\label{E:connectionAsymmetricWeak}
 \int_{t}^{+\infty}\ctr{j}{k}>0
\qquad\forall t>0
\,.
\end{equation}
\end{itemize}
Assume that the directed graph has a globally reachable node. Then it holds $x_{-}^{*}=x_{+}^{*}$.
\end{lemma}

\begin{proof}
Consider the set $I_{+}(t)$ of indices $i$ satisfying $x_{i}(t)=x^{*}_{+}$. {{}By Lemma~\ref{L:belowmaximumpreserved}, if $h>0$ then $I_{+}(t+h) \subseteq I_{+}(t)$: as time increases, the set $I_{+}(t)$ can only get smaller or remain equal}. Since it is discrete and never empty, there is some index $j_1$ with $x_{j_1}(t)=x^{*}_{+}$ for all $t\geq0$. We denote by $I_+^*$ the set of indices that meet this condition.

By hypothesis, $G$ has a globally reachable node $\ell^{*}$ and a path $j_1\to j_2\to\ldots \to j_n=\ell^*$. We now prove that $x_{\ell^*}(t)=x^{*}_{+}$ for all $t\geq0$, i.e.~$\ell^*\in I_+^*$. By contradiction, assume that $\ell^*\not\in I_+^*$. Since $j_1\in I_+^*$, in the path $j_1\to j_2\to\ldots \to j_n=\ell^*$, there exist two consecutive elements $j_{r-1}\to j_r$ such that $j_{r-1}\in I_+^*$ and $j_{r}\not \in I_+^*$. To simplify the notation, relabel indices and assume from now on $1\in I_+^*$, $2\not \in I_+^*$ and $1\to 2$.

Since $2\not \in I_+^*$, there exists $T>0$ such that $x_2(t)<x_+^*$ for all $t\geq T$, due to Lemma~\ref{L:belowmaximumpreserved}. Moreover, the existence of the arrow $1\to 2$ given by  {{}property}~\eqref{E:connectionAsymmetricWeak} ensures that it holds $\int_T^{+\infty} \ctr{1}{2}>0$, which in turn ensures that there exists $S>0$ such that $\int_T^S \ctr{1}{2}>0$. By continuity of $x_2(t)$, set $\tilde x=\max x_2([T,S])$, so that $x_2(t)\leq \tilde x$ on $[T,S]$. We now evaluate the dynamics of $x_1$ on the time interval $[T,S]$. Recalling that $x_1(t)=x^{*}_{+}$ for all $t\in[0,+\infty)$, by~\eqref{E:basicsystem} it holds
\begin{eqnarray*}
0&=&x_1(S)-x_1(T)=\sum_{k=1}^N \int_T^S \ctr{1}{k}(t)(x_k(t)-x_1(t))\,dt
=\sum_{k=1}^N \int_T^S \ctr{1}{k}(t)(x_k(t)-x^{*}_{+})\,dt\\
&\leq&\sum_{k\neq 2} 0+\int_T^S \ctr{1}{2}(t)(x_2(t)-x^{*}_{+})\,dt
\leq \int_T^S \ctr{1}{2}(t)\,dt~\cdot~ (\tilde x-x^{*}_{+})<0.
\end{eqnarray*}
This is a contradiction. Then, it holds $x_{\ell^*}(t)=x^{*}_{+}$ for all $t\geq0$. By the same reasoning with the minimum value $x^*_-$, we see that the same index $x_{\ell^*}$ satisfies $x_{\ell^*}(t)=x^{*}_{-}$ for all $t\geq0$. This implies $x^{*}_{+}=x^{*}_{-}$, which ensures $x_j(t)=x^{*}_{+}=x^{*}_{-}$ for all indices $j=\{1,\ldots,N\}$ and times $t\geq 0$.
\end{proof}
\begin{remark} The graph $G$ built in Lemma~\ref{L:constantevolution} has more connections than the one built in Theorem~\ref{T:asymmetric}, since  {{}property}~\eqref{E:connectionAsymmetricWeak} is weaker than~\eqref{E:connectionAsymmetric}. Then, requiring connectedness of $G$ is weaker than requiring connectedness of the graph in Theorem~\ref{T:asymmetric}. This weaker requirement is complemented by requiring that minimum and maximum values are constant in time.
\end{remark}

We are now ready to prove Theorem~\ref{T:asymmetric}.

\noindent {\it Proof of Theorem~\ref{T:asymmetric}:} We first observe that Proposition~\ref{p-multid} allows us to study consensus for the case $d=1$ only: we thus prove the theorem with $x_j(t)\in\R$ from now on. The structure of the proof is as follows: we first build a limit trajectory (Step 1), then prove that such a trajectory is at consensus (Step 2). We finally prove that the original dynamics converges to consensus (Step 3).

{\bf Step 1: Construction of a limit trajectory.} Let $t_n\to+\infty$ be a sequence satisfying the hypothesis of the theorem: there is a node $\ell^{*}$ such that for all $j\in\{1,\dots,N\}$ the graph ${}G=G(\{t_n\},\{\ctr{j}{k}^{}\})$ includes a directed path from $j$ to $\ell^*$. We assume that for each pair $j,k\in\{1,\dots,N\}$ the function $t\mapsto \ctr{j}{k}(t_{n}+t)$ converges to the limit function $\ctr{j}{k}^{*}$ as in Definition~\ref{D:conver}. We remark that in the hypothesis we require convergence for all pairs $j,k$ to some $\ctr{j}{k}^{*}$, eventually not satisfying~\eqref{E:connectionAsymmetric}, not only for the pairs with an arrow in the graph: when $\ctr{j}{k}$ is bounded, such limits are granted by Remark~\ref{rem-weakstar} and Banach-Alaoglu theorem, up to subsequence.


If {{}property}~\eqref{E:connectionAsymmetric} is satisfied for a pair $(j,k)$ with the original control $\ctr{j}{k}$, then the new control $\{\ctr{j}{k}^{*}\}$ satisfies  {{}property}~\eqref{E:connectionAsymmetricWeak} given in Lemma~\ref{L:constantevolution}. This implies that for each arrow $j\to k$ in the graph {{}$G=G(\{t_n\},\{\ctr{j}{k}^{}\})$ constructed} in Theorem~\ref{T:asymmetric}, the same arrow $j\to k$ exists in the graph $G^*{}=G^*(\{\ctr{j}{k}^*\})$ defined in Lemma~\ref{L:constantevolution} with controls $\{\ctr{j}{k}^{*}\}$. As a consequence, a globally reachable node of the directed graph $G$ is a globally reachable node of $G^*$.

Recall now that the support of solutions is compact, due to Proposition~\ref{p:contractive}. By passing to a subsequence in $t_n$, which we do not relabel, we assume that for each index $j\in\{1,\ldots,N\}$ the sequence $x_j(t_{n})$ admits a limit $\tilde x_j$. We consider these limits as the initial condition for the limit trajectory.

The limit system is then defined as follows: the dynamics is~\eqref{E:basicsystem}, its initial condition is $\tilde x_j$ for $j\in\{1,\ldots,N\}$ and controls are $\ctr{j}{k}^{*}$ for $j,k\in\{1,\ldots,N\}$. We denote with $x^*(t)$ the corresponding limit trajectory for the Cauchy problem of the limit system.

{\bf Step 2: The limit trajectory is at consensus.} We now prove that the limit trajectory built in the previous step is at consensus. First fix any $T>0$ and consider the exponential map $\Phi$ on the time interval $[0,T]$: it associates initial conditions and controls to the trajectory of~\eqref{E:basicsystem} as follows

$$\Phi:\begin{cases}
\R^{N}\times L^{1}\left([0,T],[0,+\infty)^{N^{2}}\right)\to& C^{0}([0,T];\R^{N})\\
 \left( x(0),\,\{t\to\ctr{j}{k}(t))\}_{j,k=1,\dots N}\right) \mapsto& t\to x(t).
\end{cases}$$
Observe that the dynamics is affine in the connection functions $\ctr{j}{k}$. We thus endow the space of Lebesgue integrable functions $L^{1}([0,T],[0,+\infty)^{N^{2}})$ with the weak$^{*}$-topology inherited by its identification as a subspace of finite nonnegative Borel measures, by testing with continuous functions $\varphi:[0,T]\to\R$.
Since, by absolute continuity of the corresponding measures, the measure of intervals converge, this topology also provides the convergence considered in Definition~\ref{D:conver}: see Lemma~\ref{L:topology}.
Recall that the map $\Phi$ is continuous, see e.g.~\cite[Theorem~3.1]{gauthier-kupka}; observe that linearity in the control plays a crucial role here. Then, consider the sequence $x^n([0,T])$ of trajectories of the original system $x(t)$ starting at time $t_n$ with initial data $x(t_n)$ and with controls $\ctr{j}{k}([t_n,t_n+T])$. By construction, both the initial data and the controls converge, hence the sequence $x^n([0,T])=\Phi(x(t_n),\ctr{j}{k}([t_n,t_n+T]))$ converges by continuity of $\Phi$. The uniqueness of the limit in $C^{0}([0,T];\R^{N})$ implies that the limit is in fact the limit trajectory $x^*$ defined above, restricted to the time interval $[0,T]$.

We now recall that the function $x_+(t):=\max_{j} x_j(t)$ is non-increasing, due to Proposition~\ref{p:contractive}, thus it admits a limit as $t\to+\infty$. 
By construction of Step 1, for the original trajectory $\lim_{n\to +\infty} x_+(t_n)=\max_j\tilde x_j$, thus by monotonicity this value is the limit of the whole trajectory $x_+(t)$.
By continuity of the map $\Phi$ and of the maximum function, the maximum function for the limit trajectory $x_+^*(t):=\max_j x_j^*(t)$ in the time interval $[0,T]$ is the uniform limit of the maximum function $x_+(t)$ on the time intervals $[t_n,t_n+T]$. As a consequence, it holds $x_+^*(0)=x_+^*(T)=\max_j\tilde x_j$.

Observe that the identity above holds for all $T>0$. This implies that the function $x_+^*(t)$ is a constant, that we denote with $x^{**}$. The same statement can be proved for the minimum function $x_-^*(t)$. Then, the limit trajectory $x^*(t)$ satisfies all the hypotheses of Lemma~\ref{L:constantevolution}, so that it holds $x_+^*(0)=x_-^*(0)=x^{**}$. As a consequence, it holds $x_j^*(0)=x^{**}$ for all $j=\{1,\ldots,N\}$.

{\bf Step 3: The original trajectory converges to consensus.} We now prove that the original trajectory $x(t)$ converges to consensus. 
Recall that by construction in Step 1 it holds $\lim_{n\to+\infty}x_j(t_n)=x_j^*(0)$, thus by Step 2 $\lim_{n\to+\infty}x_j(t_n)=x^{**}$ independent on $j$. This implies that for all $\varepsilon>0$ there exists $n^*\in\mathbb{N}$ such that $\abs{x_j(t_{n^*})-x^{**}}<\varepsilon$ for all $j\in\{1,\ldots,N\}$. By recalling that the support is contractive, due to Proposition~\ref{p:contractive}, it also holds $\abs{x_j(t)-x^{**}}<\varepsilon$ for all $j\in\{1,\ldots,N\}$ and $t\geq t_{n^*}$. This coincides with $\lim_{t\to+\infty}x_j(t)=x^{**}$ for all $j\in\{1,\ldots,N\}$. \hfill \mbox{\rule[0pt]{1.3ex}{1.3ex}}

\subsection{Proof of Corollary~\ref{c-1}} \label{s-equivalence}

In this section, we prove Corollary~\ref{c-1}. The proof is based on proving some useful equivalent formulations connected to the hypotheses of Theorem~\ref{T:asymmetric}. This also allows to better appreciate the connections with existing conditions, including persistent excitation and integral scrumbling coefficients conditions, and the novelty of our result, see \S~\ref{S:exComp} for comparisons.

\begin{lemma}\label{L:equivalentConditions}
Let $a:\R^+\to[0,+\infty)$ be Lebesgue measurable.
The following {{}properties} are equivalent:
\begin{Aenum}
\item\label{item:A} 
$\displaystyle{\limsup_{T\to+\infty}\,\liminf_{t\to+\infty}\int_{t}^{t+T}a>0}$.
\item\label{item:B} There exist $T,\mu>0$ such that for all $t\geq 0$ it holds
\begin{align} 
 \qquad\label{E:equiv2}
\int_{t}^{t+T} a\geq \mu
\,.\end{align}
\item\label{item:C} There exist $ T,\mu>0$ and a sequence $t_n\to+\infty$ with $\{ {t_{n+1}}-t_{n}\}_{n\in\N}$ bounded such that
\begin{align}
\label{E:equiv}
\int_{t_n}^{t_n+T} a\geq\mu
\qquad\forall n\in\N\,.
\end{align}
\end{Aenum}
If $ a:\R^+\to[0,+\infty)^{d}$ is bounded and all components $a_{i}$ satisfy one of the {{}properties} above, then the following weaker  {{}property} holds:
\begin{Aenum}
\setcounter{enumi}{3}
\item \label{item-D} There is a sequence $t_{n}\to+\infty$ for which the function $t\mapsto a(t_{n}+t)$ converges as in Definition~\ref{D:conver} to $ a^{*}$ with
\begin{equation}
 \int_{t}^{+\infty}a_{i}^{*}>0
\qquad
 \qquad \forall i=1\,,\dots\,,d\,
\qquad\forall t>0
\,.
\end{equation}
\end{Aenum}
\end{lemma}

\begin{proof} We first prove that Item~\ref{item:A} implies Item~\ref{item:B}. \\Set  $\ell= \limsup_{T\to+\infty}\,\liminf_{t\to+\infty}\int_{t}^{t+T}a$,
which by assumption is strictly positive.
By definition of $\ell$ as a $\limsup$, there exists $T_1>0$ with 
$\liminf_{t\to+\infty}\int_{t}^{t+T_1}a> \ell/2$.
By definition of $\liminf$ then there exists $T_2>0$ such that for all $t>T_2$ we have~\eqref{E:equiv2} with $\mu=\frac1 4\ell$ and $T=T_1$.
Choose now $T_3=\max\{T_1,T_2\}$ and observe that~\eqref{E:equiv2} is satisfied for all $t\geq T_3$ with $\mu=\frac14\ell$ and $T=T_1$. Choose now $T=2T_3$ and observe that Item~\ref{item:B} is satisfied for all $t\geq 0$ with the same $\mu$.

It is easy to prove Item~\ref{item:B} implies Item~\ref{item:C}, e.g.~by choosing $t_n=n$, $n\in\N$.

We now prove that Item~\ref{item:C} implies Item~\ref{item:A}. With no loss of generality, eventually passing to a subsequence, we assume that $t_n$ is increasing.
Set $ T_{1}= t_{1}+2\sup_{n\in \N}\{t_{n+1}-t_{n}\}$, that is finite by hypothesis. Notice that for any $T\geq T_{1}$ and any $t\geq0$ each interval $[t,t+T]$ contains some interval $[t_{n'(t)},t_{[n'+1](t)}]$ with $n'\in \N$, by construction:
thus, for all $T\geq T_1$ it also holds 
$$\liminf_{t\to+\infty}\int_{t}^{t+T}a\geq \liminf_{t\to+\infty}\int_{t_{n'(t)}}^{t_{[n'+1](t)}}a \geq \liminf_{t\to+\infty}\mu\geq \mu,$$
by monotonicity of the integral of the positive function $a$. By passing to the $\limsup$ in $T$, we have Item~\ref{item:A}.

We now prove that any of the  {{}properties} above implies Item~\ref{item-D}. We first discuss the one-dimensional case $d=1$.
We prove that Item~\ref{item:B} implies Item~\ref{item-D} when $a$ is bounded. Consider an increasing sequence $t_{n}\to+\infty$ and the corresponding sequence of translated functions $a_n:=\{t\mapsto a(t_{n}+t)\}$. It is clear that the sequence is compact in $L^\infty$ with the weak$^{*}$ topology, due to the Banach-Alaoglu theorem. By a diagonal argument we can then extract a subsequence $\{t\mapsto a(t_{n}+t)\}_{n\in\N}$ that converges to a function $a^{*}$ weakly$^{*}$ in $L^{\infty}([0,T])$, as the dual of $L^{1}([0,T])$, for all $T$, see Lemma~\ref{L:topology} for the equivalence with Definition~\ref{D:conver}. Choose the test function \(\varphi(s)=\mathbbm{1}_{[t,t+T]}(s)\). For any choice of $t>0$, we obtain Item~\ref{item-D}: by the weak$^{*}$-convergence tested with $\varphi$ and changing variable in the integral
\begin{align*}
 \int_{t}^{t+T}
\!\!\!\!\!\!\!\!\! a^{*}(s)\,ds
&=
\lim_{n\to+\infty}
\int_{t}^{t+T}
\!\!\!\!\!\!\!\!\! a(t_{n}+s)
\,ds
=
\lim_{n\to+\infty}
\int_{t_{n}+t}^{t_{n}+t+T}
\!\!\!\!\!\!\!\!\! a(s)
\,ds
\stackrel{\eqref{E:equiv2}}{\geq}\mu>0
\,.
\end{align*}

We now prove Item~\ref{item-D} for a general dimension $d>1$. First apply the proof to the first component $a_{1}$, finding a corresponding sequence $\{t^1_n\}_{n\in\N}$. Then apply the same argument to $a_{2}$, extracting a subsequence $\{t^{2}_{n_{}}\}_{n_{}\in\N}$ of $\{t^{1}_{n_{}}\}_{n_{}\in\N}$. Repeat the procedure for each component, finding a final subsequence $\{t^{d}_{n_{}}\}_{n_{}\in\N}$ for which Item~\ref{item-D} holds for all components.
\end{proof}

\begin{remark}
 It is easy to prove that Item~\ref{item-D} in Lemma~\ref{L:equivalentConditions} above is a weaker {{}property} than Items~\ref{item:A}-\ref{item:C}. Consider the sequence $t_n:=n^2$ and the $L^\infty$ function
\[a(t)=\sum_{n\in\N}\mathbbm{1}_{[n^2,n^2+n]}(t)\qquad \text{for $t\geq0$}\,,
\]
where $\mathbbm{1}_{[a,b]}$ is the indicator function of the interval.
It is clear that \(t\mapsto a(t_{n}+t)\) weakly$^{*}$ converges to $a^{*}(t)=\mathbbm{1}_{[0,+\infty)}(t)$, since each interval $[n^2,n^2+n]=[t_{n},t_{n}+n]$ in the definition of $a$ has length $n\to+\infty$. Nevertheless, observe that $$\int_{t_{n}+n}^{t_{n}+2n}a(t)\,dt=\int_{t_{n}+n}^{t_{n}+2n}0\,dt=0,$$ by observing that $n^2+n\leq n^2+2n\leq n^2+2n+1=(n+1)^2$.
This implies that \(\liminf_{t\to+\infty}\int_{t}^{t+T}a=0\) for all $T>0$, hence Item~\ref{item:A} in Lemma~\ref{L:equivalentConditions} does not hold.
\end{remark}

We are now ready to prove Corollary~\ref{c-1}.

\noindent {\it Proof of Corollary~\ref{c-1}:} {{}The proof consists in showing that Theorem~\ref{T:asymmetric} applies with $\{\ctr{j}{k}\}$ as in the statement of Corollary~\ref{c-1} and $\{\ctr{j}{k}^{*}\}$, $t_n\to+\infty$ given by Lemma~\ref{L:equivalentConditions}.

Consider indeed the following directed graphs, using that $\{\ctr{j}{k}\}$ are bounded:
\begin{itemize}
\item $G(\{\ctr{j}{k}\})$, built with one of the (equivalent) rules~\ref{item:1}-\ref{item:2}-\ref{item:3} of Corollary~\ref{c-1}.
\item $H(\{t_n\},\{\ctr{j}{k}\})$, built as in Theorem~\ref{T:asymmetric}, with $\{t_{n}\}$ given by~\ref{item-D} of Lemma~\ref{L:equivalentConditions}.
\end{itemize}

We now prove that $H$ has a globally reachable node: thus, hypotheses of Theorem~\ref{T:asymmetric} are satisfied and all solutions to~\eqref{E:basicsystem} converge to consensus..

We proved in Lemma~\ref{L:equivalentConditions} that {{}properties}~\ref{item:1}-\ref{item:2}-\ref{item:3} of Corollary~\ref{c-1} are equivalent and stronger than {{}property}~\ref{item-D} of Lemma~\ref{L:equivalentConditions}, which coincides with condition~\eqref{E:connectionAsymmetric} of Theorem~\ref{T:asymmetric}: thus graph ${}H$ contains all arrows of graph  $ G$ (and eventually some additional one). 
Since we are assuming that the directed graph $G$ has a globally reachable node, then the directed graph $H$ has a globally reachable node. \hfill \mbox{\rule[0pt]{1.3ex}{1.3ex}}}

\subsection{Proof of Theorem~\ref{T:asymmetricBis}} \label{s-asym2} In this section, we prove Theorem~\ref{T:asymmetricBis}. The proof is similar to the one for Theorem~\ref{T:asymmetric}, but replacing Lemma~\ref{L:constantevolution} with the following result.
This new lemma requires that all nodes are {{}identified with} connected in at least one direction, but connections are weaker compared to Lemma~\ref{L:constantevolution}


\begin{lemma}\label{L:constantevolutionBis}
Let $x(t)$ be a trajectory of~\eqref{E:basicsystem} in $\R$ with connection functions $\ctr{j}{k}$ in $\lin$ as in~\eqref{linloc}. Assume that both \begin{align*}
 &
x_{+}^{*}=\max\{x_{i}(t)\ :\ i=1,\dots,N\} &\text{and}
&&x_{-}^{*}=\min\{x_{i}(t)\ :\ i=1,\dots,N\}
\end{align*} are constant.
The equality $x_{-}^{*} = x_{+}^{*}$ is guaranteed {{}if} the following  {{}property} holds:
\begin{equation}\label{E:connectionAsymmetricWeakBis}
 \int_{0}^{+\infty}\ctr{j}{k}+ \int_{0}^{+\infty}\ctr{k}{j}>0
\qquad\forall j,k\in\{1\,,\dots\,,N\,.\}
\end{equation}
\end{lemma}
\begin{proof}
Consider the set $I_{+}(t)$ of indices $i$ realizing $x^{*}_{+}$. By Lemma~\ref{L:belowmaximumpreserved} the set is non-increasing in time.
Since it is discrete and never empty, there is some index, that we relabel as $1$, such that $x_1(t)=x^{*}_{+}$ for all $t\geq0$. 
Similarly, there is some index, that we relabel as $2$, such that $x_2(t)=x^{*}_{-}$ for all $t\geq0$.

Since $x_1(t)=x^{*}_{+}\geq x_j\geq x^{*}_{-}=x_2(t)$ for all $t\in[0,+\infty)$ and $j\in\{1,\dots,N\}$, then
\begin{align*}
0&=\int_{0}^{+\infty}\!\!\!\!\!\!\dot x_1
=
\sum_{j=1}^{N}\int_{0}^{+\infty} \!\!\!\!\!\ctr{1}{j}(t)\left(x_j(t)-x_1(t)\right)dt
 \leq 
 \left(\int_{0}^{+\infty}\!\!\!\!\! \ctr{1}{2}(t) dt\right)\cdot \left(x_2(t)-x_1(t)\right)\leq0\,,
 \\
0&=\int_{0}^{+\infty}\!\!\!\!\!\!\dot x_2
=
\sum_{j=1}^{N}\int_{0}^{+\infty} \!\!\!\!\!\ctr{2}{j}(t)\left(x_j(t)-x_2(t)\right)dt
\geq 
 \left(\int_{0}^{+\infty}\!\!\!\!\! \ctr{2}{1}(t) dt\right)\cdot \left(x_1(t)-x_2(t)\right)\geq0\,,
\end{align*}
where we used $x_j-x_1\leq 0$ and $x_j-x_2\geq0$ to neglect terms with the suitable sign. 
Both inequalities now read as 
$$\left(\int_{0}^{+\infty} \!\!\!\!\!\!\!\!\!\ctr{1}{2}(t) dt\right)\cdot \left(x^{*}_{-}-x^{*}_{+}\right)=\left(\int_{0}^{+\infty} \!\!\!\!\!\!\!\!\!\ctr{2}{1}(t) dt\right)\cdot \left(x^{*}_{+}-x^{*}_{-}\right)=0.$$
Thus,~\eqref{E:connectionAsymmetricWeakBis} with $(j,k)=(1,2)$ ensures what wanted: $
x_2(t)-x_1(t)=x_{-}^*-x_{+}^*=0.
$
\end{proof}

We are now ready to prove Theorem~\ref{T:asymmetricBis}.

\noindent {\it Proof of Theorem~\ref{T:asymmetricBis}.} We follow the proof of Theorem~\ref{T:asymmetric}, except for a small change in Step 2.
First, Proposition~\ref{p-multid} allows us to prove the theorem in dimension $d=1$ only. Given a trajectory $x(t)$ and the sequence $t_n\to+\infty$, we build the limit trajectory $x^*(t)$ as in Step 1 of the proof of Theorem~\ref{T:asymmetric}. We then prove that the maximal $x_+^*(t)$ and minimal values $x_-^*(t)$ of the limit trajectory are constant with respect to time, as in Step 2 of the proof of Theorem~\ref{T:asymmetric}. We now use Lemma~\ref{L:constantevolutionBis} to prove that such constant values are identical $x_+^*(t)=x_-^*(t)=x^{**}$. This in turn implies that the limit trajectory is already at consensus: $x_j^*(t)=x^{**}$. We finally prove that the original trajectory converges to consensus, as in Step 3 of the proof of Theorem~\ref{T:asymmetric}. \hfill \mbox{\rule[0pt]{1.3ex}{1.3ex}}

\section{Examples and comparison with the literature}
\label{S:exComp}

In this section, we describe some relevant examples of~\eqref{E:basicsystem}, with a double aim. First,
\begin{itemize}
 \item[\S~\ref{S:ex}:] we show that removing one hypothesis of Theorem~\ref{T:asymmetric}, Corollary~\ref{c-1} or Theorem~\ref{T:asymmetricBis} easily allows us to build counterexamples.
 \end{itemize} 
Second, we explain the novelty of our result comparing it with the literature, namely
\begin{itemize}
 \item[\S~\ref{s-PE}:] we extend Moreau, persistent excitation and integral scrambling conditions,
 \item[\S~\ref{s-cutbalance}:] our sufficient conditions are transversal to the cut-balance condition. 
 \end{itemize}

\subsection{Sharpness of hypotheses}\label{S:ex}

In this section, we show that the hypotheses of both Theorem~\ref{T:asymmetric} and Corollary~\ref{c-1} cannot be dropped, via a key counterexample.

\begin{example} \label{ex-basic2} We build the example as follows: we first define a ``building block'' on a time interval $[0,\Theta_{\eta}]$, we then iterate to concatenate controls on the whole $[0,+\infty)$. 

\paragraph{\bf Building block} We consider a system of 4 particles $(x_1,x_2,x_3,x_4)$ with initial condition $(-m,-m,m,m)$ for a given $m>0$. Fix a parameter $\eta\in(0,1)$. In $[0,\Theta_\eta]$, with $\Theta_\eta:=\log\left(\frac{4}{\eta(2-\eta)}\right)$, define all controls $\ctr{j}{k}=0$, except in the following cases:
\begin{align*}
&a\text{) for } \tau \in [0,  \log \sqrt2{} ] : \makebox[6.1cm][r]{$\ctr{1}{2}(\tau),\ctr{2}{1}(\tau),\ctr{3}{4}(\tau), \ctr{4}{3}(\tau) = 1\,,$} &\\
&b\text{) for } \tau \in [  \log 2 , \log \sqrt2{} ] : \makebox[5.6cm][r]{$\ctr{2}{3}(\tau), \ctr{3}{2}(\tau) = 1\,,$} &\\
&c\text{) for } \tau \in [\log 2 ,  \log\tfrac2 \eta] : \makebox[5.8cm][r]{$\ctr{2}{1}(\tau), \ctr{3}{4}(\tau) = 1\,,$} &\\
&d\text{) for } \tau  \in [\tfrac2 \eta,  \Theta_\eta] : \makebox[6.6cm][r]{$\ctr{1}{4}(\tau),\ctr{4}{1}(\tau) = 1\,.$} &
\end{align*}

It is easy to observe the following property of the building block: given $\eta\in(0,1)$, the time interval has length $\Theta_\eta$, that is positive and satisfies $\lim_{\eta\to 0^+}\Theta_\eta=+\infty$. The trajectory of the building block satisfies the following:
\begin{itemize}
\item Up to $\tau=\log \sqrt2{} $ the activated controls play no role on the dynamics, since $x_1=x_2=-m$ and $x_3=x_4=m$.
\item At $\tau=\log 2 $ being $\dot x_1=\dot x_4=0$ in $[ \log \sqrt2{} , \log 2 ]$ it holds $x_1(\tau)=-m$ and $x_4(\tau)=m$. Since on the second time interval it holds $\dot x_2+\dot x_3=0$ and $\dot x_2-\dot x_3=-2( x_2- x_3)$, an easy computation shows $x_2(\tau)=-\frac{m}2$, $x_3(\tau)=\frac{m}2$.
\item At $\tau=\log({2}/\eta)$ it still holds $x_1(\tau)=-m$ and $x_4(\tau)=m$. Again, an easy computation, based on the fact that $\dot x_2-\dot x_1=-(x_2-x_1)$, shows that $x_2(\tau)=-\left(1-\frac\eta2\right)m$. By a symmetry argument, we also have $x_3(\tau)=\left(1-\frac\eta2\right)m$.
\item At $\tau=\Theta_\eta$, with computations similar to those in the second time interval, we now have $x_{2}(\tau)=x_1(\tau)=-\left(1-\frac\eta2\right)m$ and $x_{3}(\tau)=x_4(\tau)=\left(1-\frac\eta2\right)m$.
\end{itemize}
In summary, in time $\Theta_\eta$, for fixed $m>0$ and $\eta\in(0,1)$, the dynamics of the building blocks steers the configuration $(-m, -m,m,m)$ to the configuration \[\left(-\left(1-\frac\eta2\right)m,-\left(1-\frac\eta2\right)m,\left(1-\frac\eta2\right)m,\left(1-\frac\eta2\right)m\right).\] See a graphical description in Figure~\ref{fig-A}.

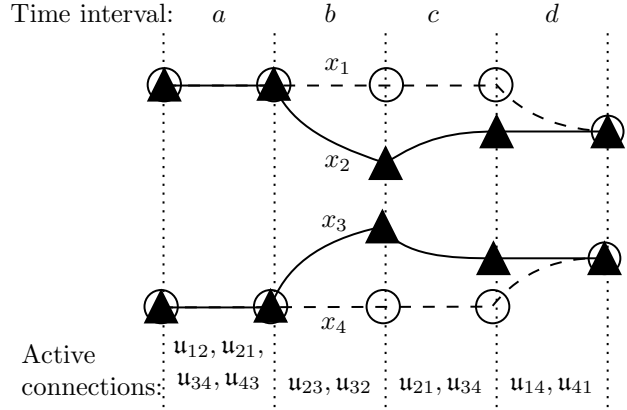
\begin{figure}[ht]
\centering
\tikzset{every picture/.style={line width=0.75pt}} 

\begin{tikzpicture}[x=0.75pt,y=0.75pt,yscale=-1,xscale=1]

\draw   (80,35.5) .. controls (80,29.7) and (84.7,25) .. (90.5,25) .. controls (96.3,25) and (101,29.7) .. (101,35.5) .. controls (101,41.3) and (96.3,46) .. (90.5,46) .. controls (84.7,46) and (80,41.3) .. (80,35.5) -- cycle ;
\draw  [fill={rgb, 255:red, 0; green, 0; blue, 0 }  ,fill opacity=1 ] (90.5,25.5) -- (100.5,45.5) -- (80.5,45.5) -- cycle ;
\draw   (149,35.5) .. controls (149,29.7) and (153.7,25) .. (159.5,25) .. controls (165.3,25) and (170,29.7) .. (170,35.5) .. controls (170,41.3) and (165.3,46) .. (159.5,46) .. controls (153.7,46) and (149,41.3) .. (149,35.5) -- cycle ;
\draw   (220,35.5) .. controls (220,29.7) and (224.7,25) .. (230.5,25) .. controls (236.3,25) and (241,29.7) .. (241,35.5) .. controls (241,41.3) and (236.3,46) .. (230.5,46) .. controls (224.7,46) and (220,41.3) .. (220,35.5) -- cycle ;
\draw   (289,35.5) .. controls (289,29.7) and (293.7,25) .. (299.5,25) .. controls (305.3,25) and (310,29.7) .. (310,35.5) .. controls (310,41.3) and (305.3,46) .. (299.5,46) .. controls (293.7,46) and (289,41.3) .. (289,35.5) -- cycle ;
\draw   (360,65) .. controls (360,59.48) and (364.48,55) .. (370,55) .. controls (375.52,55) and (380,59.48) .. (380,65) .. controls (380,70.52) and (375.52,75) .. (370,75) .. controls (364.48,75) and (360,70.52) .. (360,65) -- cycle ;
\draw  [fill={rgb, 255:red, 0; green, 0; blue, 0 }  ,fill opacity=1 ] (159.5,25) -- (169.5,45) -- (149.5,45) -- cycle ;
\draw  [color={rgb, 255:red, 0; green, 0; blue, 0 }  ,draw opacity=1 ][fill={rgb, 255:red, 0; green, 0; blue, 0 }  ,fill opacity=1 ] (230,75) -- (240,105) -- (220,105) -- cycle ;
\draw  [fill={rgb, 255:red, 0; green, 0; blue, 0 }  ,fill opacity=1 ] (300,55) -- (310,75) -- (290,75) -- cycle ;
\draw  [fill={rgb, 255:red, 0; green, 0; blue, 0 }  ,fill opacity=1 ] (370,55) -- (380,75) -- (360,75) -- cycle ;
\draw  [dash pattern={on 4.5pt off 4.5pt}]  (90.5,36) -- (299.5,36) ;
\draw    (159.5,36) .. controls (169,57) and (187,71) .. (230,85) ;
\draw    (230,85) .. controls (254,69) and (273,65) .. (300,65) ;
\draw    (90.5,36) -- (122,36) -- (159.5,36) ;
\draw    (300,65) -- (369,65) ;
\draw  [dash pattern={on 4.5pt off 4.5pt}]  (299.5,35) .. controls (317,56) and (343,64) .. (370,65) ;
\draw   (78,185.5) .. controls (78,179.7) and (82.7,175) .. (88.5,175) .. controls (94.3,175) and (99,179.7) .. (99,185.5) .. controls (99,191.3) and (94.3,196) .. (88.5,196) .. controls (82.7,196) and (78,191.3) .. (78,185.5) -- cycle ;
\draw  [fill={rgb, 255:red, 0; green, 0; blue, 0 }  ,fill opacity=1 ] (88.5,175.5) -- (98.5,195.5) -- (78.5,195.5) -- cycle ;
\draw   (147,185.5) .. controls (147,179.7) and (151.7,175) .. (157.5,175) .. controls (163.3,175) and (168,179.7) .. (168,185.5) .. controls (168,191.3) and (163.3,196) .. (157.5,196) .. controls (151.7,196) and (147,191.3) .. (147,185.5) -- cycle ;
\draw   (218,185.5) .. controls (218,179.7) and (222.7,175) .. (228.5,175) .. controls (234.3,175) and (239,179.7) .. (239,185.5) .. controls (239,191.3) and (234.3,196) .. (228.5,196) .. controls (222.7,196) and (218,191.3) .. (218,185.5) -- cycle ;
\draw   (287,185.5) .. controls (287,179.7) and (291.7,175) .. (297.5,175) .. controls (303.3,175) and (308,179.7) .. (308,185.5) .. controls (308,191.3) and (303.3,196) .. (297.5,196) .. controls (291.7,196) and (287,191.3) .. (287,185.5) -- cycle ;
\draw   (358,155) .. controls (358,149.48) and (362.48,145) .. (368,145) .. controls (373.52,145) and (378,149.48) .. (378,155) .. controls (378,160.52) and (373.52,165) .. (368,165) .. controls (362.48,165) and (358,160.52) .. (358,155) -- cycle ;
\draw  [fill={rgb, 255:red, 0; green, 0; blue, 0 }  ,fill opacity=1 ] (157.5,175) -- (167.5,195) -- (147.5,195) -- cycle ;
\draw  [color={rgb, 255:red, 0; green, 0; blue, 0 }  ,draw opacity=1 ][fill={rgb, 255:red, 0; green, 0; blue, 0 }  ,fill opacity=1 ] (228,125) -- (238,145) -- (218,145) -- cycle ;
\draw  [fill={rgb, 255:red, 0; green, 0; blue, 0 }  ,fill opacity=1 ] (298,145) -- (308,165) -- (288,165) -- cycle ;
\draw  [fill={rgb, 255:red, 0; green, 0; blue, 0 }  ,fill opacity=1 ] (368,145) -- (378,165) -- (358,165) -- cycle ;
\draw  [dash pattern={on 4.5pt off 4.5pt}]  (88.5,186) -- (297.5,186) ;
\draw    (157.5,186) .. controls (166,162) and (186,144) .. (228,135) ;
\draw    (228,135) .. controls (241,154) and (271,155) .. (298,155) ;
\draw    (88.5,186) -- (157.5,186) ;
\draw    (299,155) -- (368,155) ;
\draw  [dash pattern={on 4.5pt off 4.5pt}]  (297.5,185) .. controls (316,162) and (341,154) .. (368,155) ;
\draw  [dash pattern={on 0.84pt off 2.51pt}]  (90,13) -- (90,220) ;
\draw  [dash pattern={on 0.84pt off 2.51pt}]  (160,13) -- (160,220) ;
\draw  [dash pattern={on 0.84pt off 2.51pt}]  (230,13) -- (230,220) ;
\draw  [dash pattern={on 0.84pt off 2.51pt}]  (300,13) -- (300,220) ;
\draw  [dash pattern={on 0.84pt off 2.51pt}]  (370,13) -- (370,220) ;

\draw (200,44) node    {$x_{1}$};
\draw (200,82) node    {$x_{2}$};
\draw (198,136) node    {$x_{3}$};
\draw (197.75,180) node    {$x_{4}$};
\draw (127.22,19) node [anchor=south] [inner sep=0.75pt]    {$a$};
\draw (197.22,19) node [anchor=south] [inner sep=0.75pt]    {$b$};
\draw (263.22,19) node [anchor=south] [inner sep=0.75pt]    {$c$};
\draw (337.22,19) node [anchor=south] [inner sep=0.75pt]    {$d$};
\draw (46.73,20) node [anchor=south] [inner sep=0.75pt]   [align=left] {Time interval};
\draw (1,195) node [anchor=north west][inner sep=0.75pt]   [align=left] {Active\\connections};
\draw (97,190) node [anchor=north west][inner sep=0.75pt]    {$\begin{array}{c}\ctr{1}{2},\ctr{2}{1},\\ \ctr{3}{4},\ctr{4}{3}\end{array}$};

\draw (195,225) node [anchor=south ][inner sep=0.75pt]    {$\ctr{2}{3},\ctr{3}{2}$};
\draw (265,225) node [anchor=south ][inner sep=0.75pt]    {$\ctr{2}{1},\ctr{3}{4}$};
\draw (335,225) node [anchor=south][inner sep=0.75pt]    {$\ctr{1}{4},\ctr{4}{1}$};

\end{tikzpicture}
\caption{Example~\ref{ex-basic2}, building block.}
\label{fig-A}
\end{figure}

\paragraph{\bf Complete dynamics} Fix $m_0=1$, i.e.~start with the initial configuration $(-1,-1,1,1)$. Apply the building block dynamics by choosing the sequence $\eta_n:=2(1-\exp(-1/(n+1)^2))$ starting at $n=1$. The total length time of the time intervals up to the $n$-th buiding block is $\Theta'_n:=\sum_{j=1}^n \Theta_{\eta_j}$. Observe that the system satisfies $$x(\Theta'_n)=(-m_n,-m_n,m_n,m_n)\,,$$ with~~$m_n=m_0\Pi_{j=1}^n \left(1-\frac{\eta_j}{2}\right)=1\cdot \Pi_{j=2}^n \exp(-1/j^2)$, where $\Pi$ denotes the product of the sequence. We now prove that the system does not converge to consensus. Indeed, first observe that the concatenation of building blocks defines a trajectory on $[0,+\infty)$, since $\lim_{n\to+\infty} \Theta'_n\geq \lim_{n\to+\infty} \Theta_{\eta_n}=+\infty$, due to the fact that $\lim_{n\to+\infty}\eta_n=0$. Second, observe that it holds
\begin{eqnarray*}
\log(m_n)=-\sum_{j=2}^n 1/j^2\geq-\sum_{j=1}^{+\infty}1/j^2=-\pi^2/6.
\end{eqnarray*}
This implies $x_1(\Theta'_n)=x_2(\Theta'_n)\leq -\exp(-\pi^2/6)$ and $x_3(\Theta'_n)=x_4(\Theta'_n)\geq \exp(-\pi^2/6)$. Thus, the system does not converge to consensus.
\end{example}
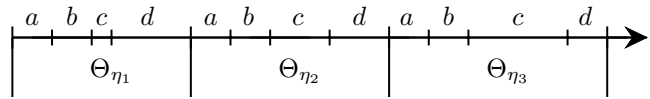
\begin{figure}[ht]
\centering
\tikzset{every picture/.style={line width=0.75pt}} 

\begin{tikzpicture}[x=0.75pt,y=0.75pt,yscale=-1,xscale=1]

\draw    (70,18) -- (387,18) ;
\draw [shift={(390,18)}, rotate = 180] [fill={rgb, 255:red, 0; green, 0; blue, 0 }  ][line width=0.08]  [draw opacity=0] (10.72,-5.15) -- (0,0) -- (10.72,5.15) -- (7.12,0) -- cycle    ;
\draw    (70,18) -- (120,18) (90,14) -- (90,22)(110,14) -- (110,22) ;
\draw [shift={(70,18)}, rotate = 180] [color={rgb, 255:red, 0; green, 0; blue, 0 }  ][line width=0.75]    (0,5.59) -- (0,-5.59)   ;
\draw    (80,18) -- (160,18) (120,14) -- (120,22) ;
\draw    (160,18) -- (220,18) (180,14) -- (180,22)(200,14) -- (200,22) ;
\draw [shift={(160,18)}, rotate = 180] [color={rgb, 255:red, 0; green, 0; blue, 0 }  ][line width=0.75]    (0,5.59) -- (0,-5.59)   ;
\draw    (260,18) -- (310,18) (280,14) -- (280,22)(300,14) -- (300,22) ;
\draw [shift={(260,18)}, rotate = 180] [color={rgb, 255:red, 0; green, 0; blue, 0 }  ][line width=0.75]    (0,5.59) -- (0,-5.59)   ;
\draw    (310,18) -- (380,18) (350,14) -- (350,22) ;
\draw    (370,18) -- (380,18)  ;
\draw [shift={(370,18)}, rotate = 180] [color={rgb, 255:red, 0; green, 0; blue, 0 }  ][line width=0.75]    (0,5.59) -- (0,-5.59)   ;
\draw    (190,18) -- (270,18) (230,14) -- (230,22) ;
\draw    (70,7) -- (70,38) ;
\draw    (160,7) -- (160,38) ;
\draw    (260,7) -- (260,38) ;
\draw    (370,7) -- (370,38) ;

\draw (78.22,15) node [anchor=south] [inner sep=0.75pt]    {$a$};
\draw (101.22,16) node [anchor=south] [inner sep=0.75pt]    {$b$};
\draw (116.22,15) node [anchor=south] [inner sep=0.75pt]    {$c$};
\draw (142.22,15) node [anchor=south] [inner sep=0.75pt]    {$d$};
\draw (170.22,16) node [anchor=south] [inner sep=0.75pt]    {$a$};
\draw (191.22,16) node [anchor=south] [inner sep=0.75pt]    {$b$};
\draw (216.22,16) node [anchor=south] [inner sep=0.75pt]    {$c$};
\draw (247.22,16) node [anchor=south] [inner sep=0.75pt]    {$d$};
\draw (270.22,16) node [anchor=south] [inner sep=0.75pt]    {$a$};
\draw (291.22,16) node [anchor=south] [inner sep=0.75pt]    {$b$};
\draw (327.22,15) node [anchor=south] [inner sep=0.75pt]    {$c$};
\draw (360.22,16) node [anchor=south] [inner sep=0.75pt]    {$d$};
\draw (106,20.4) node [anchor=north west][inner sep=0.75pt]    {$\Theta _{\eta _{1}}$};
\draw (201,20.4) node [anchor=north west][inner sep=0.75pt]    {$\Theta _{\eta _{2}}$};
\draw (301,20.4) node [anchor=north west][inner sep=0.75pt]    {$\Theta _{\eta _{3}}$};

\end{tikzpicture}
\caption{Example~\ref{ex-basic2}, complete dynamics.} 
\label{fig-B}
\end{figure}

We now discuss why Theorem~\ref{T:asymmetric} and Corollary~\ref{c-1} do not apply in this example. With this goal, we build three graphs, according to different rules:
\begin{itemize}
\item The ``unbounded interaction graph'': add~$i\to j$ if $\int_0^{+\infty}a_{ij}=+\infty$. This graph has been discussed e.g.~in~\cite{moreau2004stability,moreau2005discrete}. In this case, the graph has arrows $\{12,21,23,32,34,43\}$. The directed graph then has a globally reachable node $\ell^*$, equal to $1$ or $4$. Yet, it is well-known that (several different concepts of) connectivity of such graph do not ensure convergence. It is remarkable to observe that, in case of symmetric controls (i.e.~$\ctr{k}{j}(\tau)=\ctr{j}{k}(\tau)$ for all $j,k$ and $\tau$), the connectivity of this graph is indeed a necessary and sufficient condition to ensure convergence for all initial conditions, see~\cite[Thm 1-(c)]{hendrickx2012convergence}.

\item The graph built according to Theorem~\ref{T:asymmetric}, choosing $t_n=\Theta'_n$. 
We then have that the sequences $\ctr{j}{k}(t_n+\tau)$ weakly converge to $\ctr{j}{k}^*(\tau)$ {{}constructed} as follows:
\begin{itemize}
 \item $ \ctr{1}{2}^* =\ctr{4}{3}^*=\mathbbm{1}_{[0, \log \sqrt2{} ]}$;
 \item $\ctr{2}{3}^*=\ctr{3}{2}^*=\mathbbm{1}_{[ \log \sqrt2{} ,\log 2 ]}$;
 \item $\ctr{2}{1}^*=\ctr{3}{4}^*=\mathbbm{1}_{[0, \log \sqrt2{} ]}+\mathbbm{1}_{[\log 2 ,+\infty)}$;
\end{itemize}
 all other $\ctr{j}{k}^*$ are zero.
The graph has nodes $\{21,34\}$ only: it is not connected.

\item The graph built according to Corollary~\ref{c-1}. By taking the same sequence $t_n=\Theta'_n$ of the previous case, we have that controls satisfy {{}property}~\ref{item:C} with $T=2$ and $\mu= \log \sqrt2{} $. Indeed, we have the following:
\begin{itemize}
 \item on the time interval $[t_n,t_n+1]$, for $n$ large it holds $\ctr{1}{4}(\tau)=\ctr{4}{1}(\tau)=0;$
 \item on the time interval $[t_n,t_n+ \log \sqrt{2} ]=[t_n,t_n+\mu ]$ it holds $$\int_{t_n}^{t_n+\mu}\!\!\!\!\!\!\ctr{2}{1}(\tau)\,d\tau=
 \int_{t_n}^{t_n+\mu}\!\!\!\!\!\!\ctr{1}{2}(\tau)\,d\tau=
 \int_{t_n}^{t_n+\mu}\!\!\!\!\!\!\ctr{4}{3}(\tau)\,d\tau=
 \int_{t_n}^{t_n+\mu}\!\!\!\!\!\!\ctr{3}{4}(\tau)\,d\tau=\mu;$$
 \item on the time interval $[t_n+ \log \sqrt2{} ,t_n+\log 2 ]$ it holds $$\int_{t_n+\log \sqrt2{} }^{t_n+\log 2 }\ctr{2}{3}(\tau)\,d\tau=\int_{t_n+\log \sqrt2{} }^{t_n+\log 2 }\!\!\!\!\!\!\!\!\!\!\!\!\!\!\ctr{3}{2}(\tau)\,d\tau=\log \sqrt2{} =\mu;$$
 \item by observing that $\lim_{\eta\to 0^+}-\log(\eta)=+\infty$, on the time interval 
 $[t_n+\log 2 ,t_n+1+ \log 2 ]$ it holds $$\int_{t_n+\log 2 }^{t_n+1+\log 2 }\!\!\!\!\!\!\!\!\!\!\!\!\!\!\ctr{2}{1}(\tau)\,d\tau=\int_{t_n+\log 2 }^{t_n+1+\log 2 }\!\!\!\!\!\!\!\!\!\!\!\!\!\!\ctr{3}{4}(\tau)\,d\tau=1\geq\mu.$$
\end{itemize}
Then, the graph ${{}G=G(\{\ctr{j}{k}\})}$ built according to Corollary~\ref{c-1}, if $t_{n+1}-t_n$ was bounded, has arrows $\{21,12,23,32,34,43\}$: {the directed graph has a the arrow $\ell^*$ equal to $1$} or $4$. It is strongly connected, and even symmetric.
Yet, hypotheses of the corollary are not satisfied and the system does not converge to consensus, since the sequence $t_{n+1}-t_n$ is unbounded. Indeed, it holds
\begin{align*}
 &t_{n+1}-t_n= \Theta'_{ {n+1}}-\Theta'_{ {n}}=\Theta_{\eta_{n+1}} =\log\left(\frac{2\exp\left( {1}/{(n+1)^2}\right)}{1-\exp\left(-1/(n+1)^2\right)}\right)=\\
 &\log\left(2 (n+1)^2+o(n^2)\right)=2\log\left(n\right)+o(\log\left(n\right)).
\end{align*}
This shows that that the sequence $t_{n+1}-t_n$ is unbounded, but its growth rate is of order $2\log(n)$, that is, very slow.
\end{itemize}

We now show that Theorem~\ref{T:asymmetric} is not applicable to subsets of agents.

\begin{example} We consider a system $x$ of 6 particles with initial condition \[x(0)=(-3,-2,-2,2,2,3).\] Similarly to Example~\ref{ex-basic2}, define all controls $\ctr{j}{k}=0$, except in the following cases:
\begin{align*}
&\text{for } \tau \in [n, n+\log \sqrt2{} ] : \makebox[5cm][r]{$ \ctr{3}{4}(\tau) = \ctr{4}{3}(\tau) = 1\,,$} &n\in\N\cup\{0\},\\
&\text{for } \tau \in [n+\log \sqrt2{} , n+\log 2 ] : \makebox[3.8cm][r]{$\ctr{3}{1}(\tau) = \ctr{4}{6}(\tau) = 1\,.$}
\end{align*}
Similarly to Example~\ref{ex-basic2}, we compute 
\[
x\left(n+\log \sqrt2{} \right)=(-3,-2,-1,1,2,3) \text{ and }x\left(n+\log 2 \right)= (-3,-2,-2,2,2,3) .\]

The graph ${}G=G(\{n\},\{\ctr{j}{k}^{}\})$ of Theorem~\ref{T:asymmetric} then has nodes $\{34,43,31,46\}$ only. It is interesting to observe that the subgraph of ${}G$ with indices $\{3,4\}$ and arrows $\{34,43\}$ is complete, thus strongly connected, hence ${}G$ satisfies the hypotheses of Theorem~\ref{T:asymmetric}. Yet the corresponding subset of agents $\{3,4\}$ does not converge to consensus. In other terms, Theorem~\ref{T:asymmetric} cannot be applied to subsets of agents.
\end{example}

We now provide an example where Theorem~\ref{T:asymmetricBis} applies, while other conditions discussed here (Theorem~\ref{T:asymmetric}, Moreau, cut-balance) do not.

\begin{example} We consider a system of 3 particles with initial condition $(-1,0,1)$. 
Consider for $n\in\N$ a sequence $t_{n}\uparrow+\infty$ with $t_{n+1}-t_{n}\geq 6$, for example $t_{n}=\exp(\exp(n))$ or $t_{n}=6n$.
Similarly to Example~\ref{ex-basic2}, define all controls $\ctr{j}{k}$ arbitrarily, but nonnegative and bounded, except the following cases that we prescribe:
\begin{equation}\label{ex-3}
\begin{cases}
 \ctr{1}{2}(\tau) = 1 &\text{~~for } \tau \in [t_{n},t_{n}+1],\\
 \ctr{1}{3}(\tau) = 1 &\text{~~for } \tau \in [t_{n}+2,t_{n}+3],\\
\ctr{2}{3}(\tau) = 1 & \text{~~for } \tau \in [t_{n}+4,t_{n}+5].
\end{cases}
\end{equation}
Limit connections satisfy $\ctr{1}{2}^{*}\geq\mathbbm{1}_{[0,1]}$, $\ctr{1}{3}^{*}\geq\mathbbm{1}_{[2,3]}$, $\ctr{2}{3}^{*}\geq\mathbbm{1}_{[4,5]}$.
The graph ${}G(\{\ctr{j}{k}^{*}\})$ of Theorem~\ref{T:asymmetricBis} then has at least nodes $\{12,13,23\}$, thus Theorem~\ref{T:asymmetricBis} yields consensus.
\end{example}
\begin{remark}\label{R:eccoloqui}
 The key observation here is that Theorem~\ref{T:asymmetricBis} ensures convergence, even though we have no know about many of the controls $\ctr{j}{k}$, i.e those not defined in~\eqref{ex-3}. 
 If $t_{n+1}-t_{n}$ is bounded, also Theorem~\ref{T:asymmetric} applies, whatever the non-specified, bounded, connections are.
 If $t_{n+1}-t_{n}\to+\infty$ and if coefficients not specified by~\eqref{ex-3} vanish, then {{}$G=G(\{t_n\},\{\ctr{j}{k}^{}\})$} of Theorem~\ref{T:asymmetric} has no arrow and the Theorem~\ref{T:asymmetric} does not ensure consensus. If coefficients not specified by~\eqref{ex-3} vanish, with the choice $S=\{1\}$ the cut balance condition~\eqref{e-cutbalance} fails, as the right hand side vanishes.
\end{remark}

We finally provide an example with unbounded connections.

\begin{example} We consider a system $x$ of 3 particles with initial condition $(-1,0,1)$. 
Consider for $n\in\N$ a sequence $t_{n}\uparrow+\infty$ with $t_{n+1}-t_{n}\geq 6$, for example $t_{n}=\exp(\exp(n))$ or $t_{n}=6n$.
Similarly to Example~\ref{ex-basic2}, define all controls $\ctr{j}{k}$ arbitrarily, but nonnegative and bounded, except the following cases that we prescribe:
\begin{equation}\label{ex-unbounded}
\begin{cases}
 \ctr{1}{2}(\tau) = \tfrac{1}{\sqrt{\tau-t_{n}}}-1 &\text{~~for } \tau \in [t_{n},t_{n}+1],\\
 \ctr{1}{3}(\tau) = 1 &\text{~~for } \tau \in [t_{n}+2,t_{n}+3],\\
\ctr{2}{3}(\tau) = \tfrac{1}{\sqrt[3]{t_{n}+5-\tau}}-1& \text{~~for } \tau \in [t_{n}+4,t_{n}+5].
\end{cases}
\end{equation}

Limit connections satisfy $\ctr{1}{2}^{*}(t)\geq (\frac{1}{\sqrt{t}}-1)\mathbbm{1}_{[0,1]}$, $\ctr{1}{3}^{*}\geq\mathbbm{1}_{[2,3]}$, $\ctr{2}{3}^{*}(t)\geq (\frac{1}{\sqrt[3]{5-t}}-1)\mathbbm{1}_{[4,5]}$.
The graph {{}$G=G(\{t_n\},\{\ctr{j}{k}^{}\})=G(\{\ctr{j}{k}^{*}\})$} of Theorem~\ref{T:asymmetricBis} then has at least nodes $\{12,13,23\}$ so that Theorem~\ref{T:asymmetricBis} applies, granting convergence to consensus. If $t_{n+1}-t_{n}\to+\infty$ and if coefficients not specified by~\eqref{ex-unbounded} vanish, then {{}$G(\{t_n\},\{\ctr{j}{k}^{}\})$} of Theorem~\ref{T:asymmetric} has no arrow because $\ctr{1}{2}^{*}(t)= (\frac{1}{\sqrt{t}}-1)\mathbbm{1}_{[0,1]}$, $\ctr{1}{3}^{*}=\mathbbm{1}_{[2,3]}$, $\ctr{2}{3}^{*}(t) =(\frac{1}{\sqrt[3]{5-t}}-1)\mathbbm{1}_{[4,5]}$. If coefficients not specified by~\eqref{ex-unbounded} vanish, with the choice $S=\{1\}$ the cut balance condition~\eqref{e-cutbalance} fails, as the right hand side vanishes.
\end{example}

\subsection{Comparison with Moreau, Persistent Excitation, Integral Scrambling Coefficient conditions} \label{s-PE}

In this section, we compare our results with the Moreau condition, which ensures convergence of all solutions of~\eqref{E:basicsystem}. We also compare it with some stronger conditions that are discussed in the literature, namely the Persistent Excitation (PE) and the Integral Scrambling Coefficient (ISC). 

We first recall the precise definition of the conditions we study. For~\eqref{E:basicsystem}, they can be interpreted in a unified way, based on graph properties. These statements, equivalent to those in the literature but written in a different language, also highlight the chain of logical dependencies: the Moreau condition is weaker than PE. One can also prove  that the Moreau condition is weaker than ISC (see \cite{ABR-PE}), but we will prove convergence to consensus with a different approach.

First fix $T,\mu>0$ and consider a time $t\geq 0$. Define a graph $G(t)$ as follows: nodes are {{}identified with} $\{1,\ldots, N\}$ and an arrow from node $j$ to node $k$ is {{}drawn} if for all $t\geq 0$ it holds
 \begin{equation}\label{e-moreau}\int_t^{t+T} \ctr{j}{k}(\tau)\,d\tau\geq \mu.
 \end{equation}
We can now state the three conditions, \emph{that also require that $\{\ctr{j}{k}\}_{j,k=1}^{N}$ are bounded}:
\begin{itemize}
\item {\bf Moreau condition:} There exist $T,\mu>0$ such that the graph $G(t)$ given above is constant with respect to $t$ and has a globally reachable node.
\item {\bf ISC:} there exist $T,\mu>0$ such that for all $i,j\in\{1,\ldots,N\}$ with $i\neq j$ and $t\geq 0$ there exists an index $k_{ij}(t)$ such that both arrows $i\to k_{ij}(t)$ and $j\to k_{ij}(t)$ exist in $G(t)$.
\item {\bf PE:} there exist $T,\mu>0$ such that for all $j,k\in\{1,\ldots,N\}$ with $j\neq k$ the arrow $j\to k$ exists in $G(t)$ (thus also $k\to j$ and $G(t)$ must be constant).
\end{itemize}
\begin{remark}
    While Moreau and PE condition require a graph that is constant with respect to time, ISC does not require it. In the case of a finite number of agents, one can anyway adapt the Moreau condition to a time-dependent graph and prove that it is weaker than ISC, but changing the values of parameters $\mu,T$. See \cite{ABR-PE}.
\end{remark}
We now prove that the Moreau condition is equivalent to  {{}property}~\ref{item:2} of Corollary~\ref{c-1}, while ISC condition is a particular case of {{}property \eqref{E:connectionAsymmetric} in} Theorem~\ref{T:asymmetric}
; thus, our results generalize the convergence of systems under Moreau, ISC, and PE conditions proved in~\cite{bonnet2021consensus,bonnet2022consensus,moreau2004stability}.
Dropping the assumption that connections are bounded, such conditions are known to be not sufficient, see~\cite[Page 4002]{moreau2004stability}.
\begin{lemma} Consider bounded signals $\ctr{j}{k}$, which satisfy Moreau condition, or ISC, or PE. Then, all trajectories of~\eqref{E:basicsystem} converge to consensus.
\end{lemma}
\begin{proof}

Observe that the graph $G$ built with the Moreau condition~\eqref{e-moreau} and the graph $H$ built with Corollary~\ref{c-1} coincide, because~\eqref{e-moreau} is condition~\ref{item:2} in Corollary~\ref{c-1}.
Since both Moreau sufficient condition and Corollary~\ref{c-1} require a globally reachable node for such graph $G=H$, they are identical sufficient conditions.

We now observe that the PE condition corresponds to the fact that the graph built with the Moreau condition is complete. Thus, it has a globally reachable node, hence consensus occurs.

We now prove that, under the ISC condition, hypotheses of Theorem~\ref{T:asymmetric} are satisfied. For each $t\in[0,+\infty)$, denote with $\mathcal G(t)$ the graph with nodes $\{1,\ldots,N\}$ and arrows given by $i\to k_{ij}$, $j\to k_{ij}$, where $k_{ij}$ is given by the ISC condition. 

We first prove that, when $\mathcal G(t)$ is constant, ISC implies Moreau condition:

{{\bf Claim:} \it Each graph $\mathcal G(t)$ has a globally reachable node.}

\begin{proof}
Consider the operator $\Gamma$ defined as follows: given a (finite) set $A$ of $n$ distinct indexes, fix any order $A=\{i_1,\ldots,i_n\}$, and define
$$\Gamma(A):=\begin{cases}
 \{k_{i_1 i_2},k_{i_3 i_4},\ldots,k_{i_{n-1},i_n}\}&\mbox{~~ for $n$ even},\\
 \{k_{i_1 i_2},k_{i_3 i_4},\ldots,k_{i_{n-2},i_{n-1}},i_n\}&\mbox{~~ for $n$ odd},
\end{cases}$$
where $k_{i j}$ is the index given by the ISC condition. Since $\Gamma(A)$ is a set, any multiple occurrences of the same element are reduced to one. As a result, the set $\Gamma(A)$ has $\left \lceil \frac{n}2\right\rceil$ elements at most, where $\lceil x\rceil$ is the smallest integer larger than or equal to $x$.

Consider now the set $A_0:=\{1,\ldots, N\}$ of agents in~\eqref{E:basicsystem}, seen as the nodes of the graph $\mathcal G(t)$. Recursively define $A_{m+1}:=\Gamma(A_m)$, until $A_m$ is reduced to a single element, which we denote with $\ell$. The fact that the process ends is a consequence of the fact that $\Gamma(A_m)$ has fewer elements than $A_m$ as soon as $A_m$ is not reduced to a single element.  By the definition of $\Gamma$, the sets $A_m$ satisfy the following property: for each $i_m\in A_m$ there exists $i_{m+1}\in A_{m+1}$ such that an arrow $i_m\to i_{m+1}$ is in $\mathcal G(t)$. By construction, each index $i_0\in A_0$ has an index $i_1\in A_1$, then an index $i_2\in A_2$, and so on; this implies that the graph includes the directed path $i_0\to i_1\to i_2 \to \ldots\to \ell^*$. Since this property holds for any $i_0\in A_0$, i.e.~for any node in the graph, the graph $\mathcal G(t)$ has a globally reachable node.
\end{proof}

After proving the claim, we have the following key observation: each $\mathcal G(t)$ is an element of the set of simple directed graphs with $N$ nodes (i.e.~graphs in which for each ordered pair of indexes $i,j$ there exists either zero or one arrow, and no arrows from $i$ to $i$, for $i,j\in\{1,\dots,N\}$). 
It is valued in a finite set, since it is contained in the set of simple directed graphs, that has $2^{N(N-1)}$ elements.

Enumerate the image graphs $\mathcal{G}(\R^{+})=\{\mathcal{G}_{1},\dots, \mathcal{G}_{K}\}$: they have a globally reachable node by the Claim.
%
%
%
%
By the Banach-Alaoglu theorem, and by Lemma~\ref{L:topology}, there exists a subsequence $t_{n}$ of $nT$, $n\in\N$, for which all the functions $f_n(t):=\ctr{j}{k}(t_n+t)$ converge, as in Definition~\ref{D:conver}, to functions $\ctr{j}{k}^{*}$.
Being valued in a finite set, up to subsequence, that we do not relabel, we can think that $\mathcal G(t_{n})$ is constantly $\mathcal G_{m_{1}}$.
Set $s_{0}:=t_{1}$.
Extract now a subsequence, that we do not relabel, so that $\mathcal G(t_{n}+T)$ is constantly $\mathcal G_{m_{2}}$ and set $s_{1}:=t_{2}$.
At the $\ell$-th step, $\ell\in\N$, extract a subsequence, that we do not relabel, so that also $\mathcal G(t_{n}+\ell T)$ is constantly $\mathcal G_{m_{\ell}}$ and set $s_{\ell}:=t_{\ell+1}$.

Denote by $\mathcal G^{*}(t)$ the graph associated to connections $\ctr{j}{k}^{*}$ with condition~\eqref{e-moreau}.
By construction, $\mathcal G^{*}(\ell T)$ contains all the arrows in $\mathcal G_{m_{\ell}}$, for $\ell\in\N$, drawn with condition~\eqref{e-moreau} on connections $\ctr{j}{k}^{*}$.
If the sequence $m_{\ell}$ contains the index $m^{*}$ for infinitely many $\ell_{r}$, $r\in\N$, then, whenever $j\to k$ is an arrow of $\mathcal G_{m^{*}}$, for every $t>0$ it holds 
\[
\int_t^{+\infty}\ctr{j}{k}^*(s)\,ds\geq \#\{\ell_{r}\ :\ \ell_{r}T\geq t\}\cdot \mu=+\infty,
\]
where $\#A$ denotes the number of elements of the set $A$. Thus, the arrow $j\to k$ belongs to the graph ${}G=G(\{s_n\},\{\ctr{j}{k}^{}\})$ constructed as in Theorem~\ref{T:asymmetric} relative to the sequence $s_{n}$.
We proved that ${}G $ contains all arrows of $\mathcal G_{m^{*}}$, hence it admits a globally reachable node too. Then, hypotheses of Theorem~\ref{T:asymmetric} 
are satisfied, and the system converges to consensus for any initial condition.
\end{proof}

\subsection{Comparison with cut-balance conditions} \label{s-cutbalance} In this section, we compare our results with the so-called cut-balance condition, introduced in~\cite{hendrickx2012convergence,martin2015continuous} either in instantaneous or non-instantaneous setting. We recall here the most general formulation, presented in~\cite[Assumptions 1-2]{martin2015continuous}, that is as follows:
\begin{itemize}
 \item {\bf Cut-balance condition}: There exists a sequence of times $\tau_n\to+\infty$ and uniform bounds $K,M>0$ such that for all subsets $S$ of indices, the following non-instantaneous property holds:
\begin{equation}\label{e-cutbalance}
\sum_{j\in S,k\not\in S} \int_{\tau_n}^{\tau_{n+1}}\ctr{j}{k}(s)\,ds\leq K \sum_{j\in S,k\not\in S} \int_{\tau_n}^{\tau_{n+1}}\ctr{k}{j}(s)\,ds\leq M.
\end{equation}
\end{itemize}
As described by the authors, this is a reciprocity condition: the outward connections from $S$ are proportional to the inward ones, over subsequent time intervals. The property is not instantaneous since connections are measured as time integrals.

We now highlight the main difference between the hypotheses of our result and the cut-balance conditions, that is the already highlighted reciprocity condition. In our reasoning, there is no comparison between inward and outward connections. Rather on the opposite, the main connectivity hypothesis is a tree-like property. We will show this aspect with the following example.

\begin{example} \label{ex-cutbalance} Take a system of four agents $(x_1,x_2,x_3,x_4)$ that interact as follows:
\begin{enumerate}
 \item $\ctr{1}{2}=\ctr{2}{1}=\ctr{3}{4}=\ctr{4}{3}=\ctr{2}{3}=1$;
 \item $\ctr{3}{2}$ bounded and nonnegative, to be chosen later;
 \item all other $\ctr{j}{k}$ are zero. 
\end{enumerate}
It is clear that the cut-balance condition~\eqref{e-cutbalance} is satisfied for some choices of $\ctr{3}{2}$ only. Indeed, by choosing $S=\{1,2\}$ one has that the condition reads as 
$$\int_{\tau_n}^{\tau_{n+1}} \ctr{2}{3}(t)\,dt=\tau_{n+1}-\tau_n\leq K \int_{\tau_n}^{\tau_{n+1}} \ctr{3}{2}(t)\,dt.$$ This is not satisfied e.g.~for any function that satisfies $\lim_{t\to+
\infty}\ctr{3}{2}(t)=0$.

In contrast, we see that for any choice of sequence $t_n\to+\infty$, all interaction functions converge to their natural limit (with constant value 1 or 0), except for $\ctr{3}{2}$. Here, the key observation is that, due to the Banach-Alaoglu theorem, there exists a subsequence, that we do not relabel, $t_n\to+ \infty$ such that $\ctr{3}{2}(t_n+t)$ converges to some limit $\ctr{3}{2}^*(t)$. The fact that this limit satisfies~\eqref{E:connectionAsymmetric} or not plays no role in the hypotheses of Theorem~\ref{T:asymmetric}: in fact, the graph {{}$G=G(\{t_n\},\{\ctr{j}{k}^{}\})=G(\{\ctr{j}{k}^{*}\})$} already contains arrows $\{12,21,23,34,43\}$ and admits a globally reachable node $\ell^*$ equal to $3$ or $4$. Thus, the system converges to consensus for any choice of the initial data and any choice of the interaction function $\ctr{3}{2}(t)$.
\end{example}

The example above shows that, in some cases, our theorems provide convergence in cases in which the cut-balance condition is not satisfied. More interestingly, it shows that our theorems can be applied by studying a subset of pairs of indices only, in the spirit of Remark~\ref{rem-BA}: indeed, assume that, if for a choice $t_n\to+\infty$, one can prove convergence of $\ctr{j}{k}(t+t_n)$ to some $\ctr{j}{k}^*(t)$ satisfying~\eqref{E:connectionAsymmetric} just for some pairs $i\to j$ and the corresponding directed graph ${}G(\{t_n\},\{\ctr{j}{k}^{}\})$ admits a globally reachable node $\ell^*$. In this case, the convergence of the $\ctr{j}{k}(t+t_n)$ to some $\ctr{j}{k}^*(t)$ for the remaining pairs of indices is ensured, by passing to a subsequence. The actual value of such remaining $\ctr{j}{k}^*(t)$ plays no role, since it amounts to add connections to  ${}G(\{t_n\},\{\ctr{j}{k}^{}\})$, that already has a globally reachable node for sure.

There is one more difference between the hypotheses of our result and the cut-balance conditions: this is about the time intervals in which hypotheses need to be proven. In our Corollary~\ref{c-1}, {{}property}~\ref{item:C} needs to be verified on time intervals of the form $[t_n,t_n+T]$ for a given sequence $t_n\to+\infty$. In the cut-balance hypothesis, one instead needs to split the whole time interval $[0,+\infty)$ into intervals of the form $[\tau_n,\tau_{n+1}]$ and verify the condition for all times.

We finally observe that our results are somehow transversal with respect to the cut-balance condition. Indeed, there are cases in which our results do not apply, while the cut-balance condition is satisfied and it ensures convergence. We show here a simple example.

\begin{example} Consider a system of three agents $(x_1,x_2,x_3)$ with these connections:
\[
\bullet\ \ctr{1}{2}(t)=\ctr{2}{1}(t)=1
\qquad
\bullet\ \ctr{2}{3}(t)=\ctr{3}{2}(t)= (t+1)^{-1}
\qquad
\bullet\ \ctr{1}{3}(t)=\ctr{3}{1}(t)=0\,.
\] 
It is clear that the controls converge to $\ctr{1}{2}^*=\ctr{2}{1}^*\equiv 1$ and $\ctr{2}{3}^*=\ctr{3}{2}^*=\ctr{1}{3}^*=\ctr{3}{1}^*=0$, thus for any choice of $t_n\to +\infty$ the graph  ${}G(\{t_n\},\{\ctr{j}{k}^{}\})$ {{}constructed} in Theorem~\ref{T:asymmetric} contains arrows $\{12,21\}$ only. Thus, our result does not ensure convergence.

Instead, one can prove that the system satisfies the cut-balance property~\eqref{ex-cutbalance} with $K=1$, since it is symmetric. Then, the system converges to consensus.
\end{example}

\begin{remark} The example above raises an open question: by a time rescaling, one can easily transform controls $\ctr{2}{3},\ctr{3}{2}(t)$ to constant positive functions, that in turn have a natural limit satisfying {{}property}~\eqref{E:connectionAsymmetric}. This comes with the price of letting controls $\ctr{1}{2},\ctr{2}{1}$ explode, then bringing the system outside the hypotheses of Theorem~\ref{T:asymmetric}. Yet, one may read the example above as a double time-scale dynamics: while agents 1,2 have a fast interaction, agents 2,3 have a slow one. Virtually, one may say that agents 1-2 first reach consensus, then the double agent 1-2 and the single agent 3 reach consensus. We aim to address this question in a future research.
\end{remark}

The example above also highlights that results about consensus can be achieved by a time rescaling. Our statements can then be slightly generalized as follows.
\begin{corollary}\label{C:reparam}
For $j,k=1,\dots,N$, let $\ctr{j}{k}$ be that are measurable for all continuous probability measures, i.e.~``universally measurable''.
Let $\rho:\R^+\to\R^+$ be increasing, absolutely continuous, and diverging at $+\infty$.
If $\norm{\ctr{j}{k}(\rho(t),x)\dot\rho(t)}_\infty$ is finite and if Theorem~\ref{T:asymmetric}, or Theorem~\ref{T:asymmetricBis}, holds for the connection functions ${\ctr{j}{k}^-}(t)=\inf_{x}\ctr{j}{k}(\rho(t),x)\dot\rho(t)$, then any global trajectory of~\eqref{E:non-lin} converges to consensus.
\end{corollary}

Our results raise new questions about their integration with other available criteria (e.g. cut-balance) and their extension to dynamics with different time-scales (e.g. fast and slow variables).

\end{document}